\documentclass[a4paper,11pt]{article}
\usepackage[english]{babel}
\usepackage{tikz}
\usetikzlibrary{matrix,arrows,decorations.markings}
\usepackage{amsmath,amsfonts,amssymb,amsthm,url,textcomp}
\usepackage{amsfonts}
\usepackage{csquotes}
\usepackage{a4wide}
\usepackage[numbers]{natbib}
\usepackage{fancyhdr}
\usepackage{hyperref}
\usepackage{anyfontsize}

\numberwithin{theorem}{section}

\newtheorem{theoremm}{Theorem}\numberwithin{theoremm}{subsection}
\newtheorem{deffinition}[theoremm]{Definition}
\newtheorem{lemmma}[theoremm]{Lemma}
\newtheorem{corrollary}[theoremm]{Corollary}
\newtheorem{propposition}[theoremm]{Proposition}
\theoremstyle{remark}

\newtheorem{remmark}[theoremm]{Remark}
\newtheorem{exxample}[theoremm]{Example}
\newtheorem{quesstion}[theoremm]{Question}

\begin{document}

\title{Classification of finite group automorphisms with a large cycle}

\author{Alexander Bors\thanks{University of Salzburg, Mathematics Department, Hellbrunner Stra{\ss}e 34, 5020 Salzburg, Austria. \newline E-mail: \href{mailto:alexander.bors@sbg.ac.at}{alexander.bors@sbg.ac.at} \newline The author is supported by the Austrian Science Fund (FWF):
Project F5504-N26, which is a part of the Special Research Program \enquote{Quasi-Monte Carlo Methods: Theory and Applications}. \newline 2010 \emph{Mathematics Subject Classification}: primary: 20B25, 20G40, 20D45, 20K01, 20K30, secondary: 11A07, 12E20, 15A21, 20F05, 37P99. \newline \emph{Key words and phrases:} finite groups, automorphisms, cycle structure, large cycles}}

\date{March 31, 2015}

\maketitle

\begin{abstract}
Let $\psi$ be a permutation of a finite set $X$. We define $\lambda(\psi)$ to be the largest fraction of elements of $X$ lying on a single cycle of $\psi$. For a finite group $G$, we define $\lambda(G)$ to be the maximum among the values $\lambda(\alpha)$, where $\alpha$ runs through the automorphisms of $G$. In this paper, we develop tools to deal with questions related to $\lambda$-values of finite groups and of their automorphisms. As a consequence, we will be able to give a classification, up to a natural notion of isomorphism, of those pairs $(G,\alpha)$ where $G$ is a finite group, $\alpha$ is an automorphism of $G$ and $\lambda(\alpha)\geq\frac{1}{2}$.
\end{abstract}

\section{Introduction}\label{sec1}

\subsection{Some background, definitions and the main results}\label{subsec1P1}

By results of Miller \cite{Mil29b} from 1929, for any $\alpha\in\{-1,2,3\}$ and any group $G$, if the map $G\rightarrow G,g\mapsto g^{\alpha}$ is an automorphism of $G$, then $G$ is abelian (the argument for $\alpha=3$ and finite $G$ given there can easily be adapted to general $G$), whereas for all $\alpha\in\mathbb{Z}\setminus\{-1,0,1,2,3\}$, there exist finite nonabelian groups for which the power map with respect to $\alpha$ is an automorphism. For finite groups $G$, the first result can be reformulated as follows: Fix $\alpha\in\mathbb{Z}$, set, for an automorphism $\psi$ of $G$, $\mathrm{l}_{\alpha}(\psi):=\frac{1}{|G|}\cdot|\{g\in G\mid \psi(g)=g^{\alpha}\}|$, and $\mathrm{l}_{\alpha}(G):=\mathrm{max}_{\psi\in\mathrm{Aut}(G)}{\mathrm{l}_{\alpha}(\psi)}$. Then for $\alpha\in\{-1,2,3\}$, the condition $\mathrm{l}_{\alpha}(G)=1$ implies abelianity of $G$.

This led to a series of papers investigating conditions on finite groups $G$ of the form $\mathrm{l}_{\alpha}(G)>\rho$ for $\alpha\in\{-1,2,3\}$ and $\rho\in\left(0,1\right)$. Each of the conditions $\mathrm{l}_{-1}(G)>\frac{3}{4}$, $\mathrm{l}_2(G)>\frac{1}{2}$ and $\mathrm{l}_3(G)>\frac{3}{4}$ implies that $G$ is abelian, and these bounds are sharp, see \cite{Mil29a,Lie73a,Mac75a}. Restricting, for some prime $p$, one's attention to the class $\mathcal{G}_p$ of finite groups whose least prime divisor is $p$, the bounds can be improved: For $p>2$, $\alpha\in\{-1,2,3\}$ and $G\in\mathcal{G}_p$, the condition $\mathrm{l}_{\alpha}(G)>\frac{1}{p}$ implies abelianity of $G$, as does $\mathrm{l}_2(G)>\frac{1}{2}$ for $G\in\mathcal{G}_2$, see \cite{LM73a,Lie73a,Mac75a}. Classifications are available for the finite groups $G$ with $\mathrm{l}_{-1}(G)>\frac{1}{2}$ \cite{LM72a}, $\mathrm{l}_2(G)\geq\frac{5}{12}$ \cite{Zim90a} and $\mathrm{l}_3(G)=\frac{3}{4}$ \cite{Mac75a}, and also, for odd primes $p$, for the $G\in\mathcal{G}_p$ with $\mathrm{l}_{-1}(G)=\frac{1}{p}$ \cite{LM73a}, $\mathrm{l}_2(G)=\frac{1}{p}$ \cite{Lie73a} and $\mathrm{l}_3(G)=\frac{1}{p}$ \cite{DM89a} as well as for the $G\in\mathcal{G}_2$ with $\mathrm{l}_2(G)=\frac{1}{2}$ \cite{Lie73a}. Other notable results on such conditions are: All $\mathrm{l}_{-1}$-values greater than $\frac{1}{2}$ are of the form $\frac{k+1}{2k}$ for some positive integer $k$ \cite{LM72a}, $\mathrm{l}(G)>\frac{4}{15}$ implies that $G$ is solvable (and this bound is sharp) \cite{Pot88a}, and the derived length of a finite solvable group $G$ is (explicitly) bounded in terms of $\mathrm{l}_{-1}(G)$ \cite{Heg05a}.

The values $\mathrm{l}_{\alpha}(G)$ are maximum fractions of elements of the finite group $G$ showing a certain behavior under an appropriate automorphism of $G$. In this paper, we investigate such a fraction of a different nature, namely the maximum fraction of elements of $G$ lying on a single automorphism cycle.

\begin{deffinition}\label{lambdaDef}
(1) Let $X$ be a finite set, $\psi$ a permutation on $X$. We define $\Lambda(\psi)$ as the largest length of one of the disjoint cycles into which $\psi$ decomposes, and set $\lambda(\psi):=\frac{1}{|X|}\Lambda(\psi)$.

(2) Let $G$ be a finite group. We define $\Lambda(G):=\mathrm{max}_{\alpha\in\mathrm{Aut}(G)}{\Lambda(\alpha)}$ and $\lambda(G):=\frac{1}{|G|}{\Lambda(G)}$.

(3) For a fixed number $\rho\in\left(0,1\right)$, we say that a finite group $G$ satisfies the \textbf{(automorphism) $\rho$-large-cycle-condition} (\textbf{$\rho$-LCC}) if and only if $\lambda(G)>\rho$.
\end{deffinition}

Note that for nontrivial $G$, we always have $\Lambda(G)\leq|G|-1$ (and thus $\lambda(G)<1$), since every automorphism fixes the identity element (actually, Horo\v{s}evski\u{\i} \cite[Theorem 2]{Hor74a} showed that for nontrivial $G$, $|G|-1$ even is an upper bound on the order of any automorphism of $G$).

\begin{exxample}\label{lambdaEx}
(1) $\lambda$-values of primary cyclic groups: Let $p$ be an odd prime and $k\in\mathbb{N}^+$. Then $\lambda(\mathbb{Z}/p^k\mathbb{Z})=\frac{p^{k-1}(p-1)}{p^k}=1-\frac{1}{p}$, since the largest automorphism order of $\mathbb{Z}/p^k\mathbb{Z}$ is $\phi(p^k)=p^{k-1}(p-1)$ and this, as for any finite nilpotent group (see \cite[Corollary 1]{Hor74a}) coincides with the $\Lambda$-value. Similarly, we have $\lambda(\mathbb{Z}/2\mathbb{Z})=\lambda(\mathbb{Z}/4\mathbb{Z})=\frac{1}{2}$, whereas for $k\geq 3$, $\lambda(\mathbb{Z}/2^k\mathbb{Z})=\frac{1}{4}$.

(2) $\lambda$-values of finite elementary abelian groups: Let $p$ be any prime and $k\in\mathbb{N}^+$. Then $\lambda((\mathbb{Z}/p\mathbb{Z})^k)=\frac{p^k-1}{p^k}=1-\frac{1}{p^k}$, since the upper bound $p^k-1$ on the $\Lambda$-value is attained in any finite elementary abelian group. To see this, one can use the following well-known construction: Let $\mathbb{F}_{p^k}$ denote the finite field with $p^k$ elements, and let $\alpha$ be a generator of the group of units $\mathbb{F}_{p^k}^{\ast}$. Then multiplication with $\alpha$ is an automorphism of the underlying additive group (which is isomorphic to $(\mathbb{Z}/p\mathbb{Z})^k$) shifting all nonzero elements in one cycle.
\end{exxample}

This gives us examples of finite groups with $\lambda$-value in $\left[\frac{1}{2},1\right]$. By our main results, we will provide a complete answer to the following question:

\begin{quesstion}\label{ques}
Which are the finite groups $G$ with $\lambda(G)\geq\frac{1}{2}$? That is, which finite groups admit an automorphism with a cycle filling at least half of the group?
\end{quesstion}

Our results are actually more explicit than just listing the isomorphism types of finite groups $G$ with $\lambda(G)\geq\frac{1}{2}$; instead, we will give the corresponding list of \enquote{isomorphism types of group automorphisms}. More precisely, we will work with the following definition (which is from the author's preprint \enquote{On the dynamics of endomorphisms of finite groups}, see \url{http://arxiv.org/abs/1409.3756}):

\begin{deffinition}\label{fdgDef}
A \textbf{finite dynamical group} (\textbf{FDG}) is a finite group $G$ together with an endomorphism $\varphi$ of $G$. It is called \textbf{periodic} if and only if $\varphi$ is an automorphism of $G$, and the $\lambda$-value of a periodic FDG $(G,\alpha)$ is understood as the $\lambda$-value of $\alpha$. A periodic FDG $(G,\alpha)$ is called \textbf{$\lambda$-maximal} if and only if $\lambda(\alpha)=\lambda(G)$. An FDG $(G,\varphi)$ is called \textbf{(elementary) abelian} if and only if $G$ is (elementary) abelian, \textbf{nonabelian} if and only if it is not abelian, and for a prime $p$, $(G,\varphi)$ is called a \textbf{$p$-FDG} if and only if $G$ is a $p$-group. For FDGs $(G_1,\varphi_1),\ldots,(G_r,\varphi_r)$, their \textbf{product $\prod_{i=1}^r{(G_i,\varphi_i)}$} is defined as the FDG $(\prod_{i=1}^r{G_i},\prod_{i=1}^r{\varphi_i})$, where $\prod_{i=1}^r{\varphi_i}$ is the endomorphism of $\prod_{i=1}^r{G_i}$ sending $(g_1,\ldots,g_r)\mapsto(\varphi_1(g_1),\ldots,\varphi_r(g_r))$.
\end{deffinition}

FDGs can be seen as algebraic structures obtained by expanding (in the model-theoretic sense) the group structure of $G$ by an additional unary operation. As such, they have a natural notion of homomorphism and isomorphism: For FDGs $(G,\varphi)$ and $(H,\psi)$, a \textit{homomorphism between $(G,\varphi)$ and $(H,\psi)$} is a group homomorphism $f:G\rightarrow H$ such that $f\circ\varphi=\psi\circ f$, and an \textit{isomorphism between $(G,\varphi)$ and $(H,\psi)$} is a homomorphism which is a bijection between the underlying sets (whose inverse then automatically is a homomorphism between $(H,\psi)$ and $(G,\varphi)$). In particular, for two endomorphisms $\varphi_1,\varphi_2$ of a fixed finite group $G$, the FDGs $(G,\varphi_1)$ and $(G,\varphi_2)$ are isomorphic if and only if $\varphi_1$ and $\varphi_2$ are \textit{conjugate in $\mathrm{End}(G)$}, i.e., there exists an $\alpha\in\mathrm{Aut}(G)=\mathrm{End}(G)^{\ast}$ (the group of units in $\mathrm{End}(G)$) such that $\alpha\circ\varphi_1\circ\alpha^{-1}=\varphi_2$. Note that FDGs can be seen as expansions of the following:

\begin{deffinition}\label{fdsDef}
A \textbf{finite dynamical system} (\textbf{FDS}) is a finite set $S$ together with a function $f:S\rightarrow S$. It is called \textbf{periodic} if and only if $f$ is a permutation of $S$. For an element $s\in S$, the \textbf{orbit of $s$ under $f$} is the set $\{f^n(s)\mid n\in\mathbb{N}\}$. If $(S_1,f_1),\ldots,(S_r,f_r)$ are FDSs, their \textbf{disjoint union} is the FDS $(\sqcup_{i=1}^r{S_i},\sqcup_{i=1}^r{f_i})$, where $\sqcup$ denotes the set-theoretic disjoint union.
\end{deffinition}

Hence FDSs are just finite sets endowed with one unary operation, a \textit{homomorphism between FDSs $(S_1,f_1)$ and $(S_2,f_2)$} is a map $h:S_1\rightarrow S_2$ such that $h\circ f_1=f_2\circ h$, and an \textit{FDS isomorphism} is a bijective FDS homomorphism.

Our goal will be to classify, up to isomorphism, the periodic FDGs with $\lambda$-value at least $\frac{1}{2}$. Example \ref{lambdaEx} provides us with examples of such FDGs among those from Definition \ref{specialFdgDef} below. Recall that for every field $K$ and every monic polynomial $Q(X)=X^n+a_{n-1}X^{n-1}+\cdots+a_1X+a_0\in K[X]$ of degree $n$, the $(n\times n)$-matrix \[\begin{pmatrix}0 & 0 & \cdots & 0 & -a_0 \\ 1 & 0 & \cdots & 0 & -a_1 \\ 0 & 1 & \cdots & 0 & -a_2 \\ \vdots & \vdots & \cdots & \vdots & \vdots \\ 0 & 0 & \cdots & 1 & -a_{n-1}\end{pmatrix}\] with entries in $K$ is called the \textit{companion matrix of $Q(X)$}, denoted by $\mathrm{M}(Q(X))$.

\begin{deffinition}\label{specialFdgDef}
(1) Let $m\in\mathbb{N}^+$ and $a\in\mathbb{N}$. Then $\mathcal{M}(m,a)$ denotes the FDG $\mathbb{Z}/m\mathbb{Z}$ together with the endomorphism sending each element to its $a$-fold.

(2) Let $m,d\in\mathbb{N}^+$ and let $A$ be a $(d\times d)$-matrix over the ring $\mathbb{Z}/m\mathbb{Z}$. Then $\mathcal{V}(m,A)$ denotes the FDG $(\mathbb{Z}/m\mathbb{Z})^d$, viewed as the free $\mathbb{Z}/m\mathbb{Z}$-module of rank $d$, together with the endomorphism represented with respect to the standard basis by the matrix $A$.

(3) Let $p$ be a prime and let $P(X)\in\mathbb{F}_p[X]$ be monic. We set $\mathcal{V}(P(X)):=\mathcal{V}(p,\mathrm{M}(P(X)))$.
\end{deffinition}

The notation from Definition \ref{specialFdgDef} is motivated by names of classical pseudorandom number generators: $\mathcal{M}$ originates from \enquote{multiplicative congruential generator} and $\mathcal{V}$ from \enquote{vector generator}; see Subsection \ref{subsec4P1} for some more connections of our results with pseudorandom number generation. It is now time to present the main results of this paper, Theorem \ref{abelianTheo}, Corollary \ref{strictCor}, Theorem \ref{equalTheo} and Theorem \ref{denseTheo}:

\begin{theoremm}\label{abelianTheo}
Finite groups $G$ with $\lambda(G)>\frac{1}{2}$ are abelian.
\end{theoremm}

This bound is sharp, since there exist finite nonabelian groups whose $\lambda$-value equals $\frac{1}{2}$, for example finite dihedral groups (see also Theorem \ref{equalTheo}). Using Theorem \ref{abelianTheo} as well as methods for discussing LCCs in finite abelian groups, we will be able to show the classification result Corollary \ref{strictCor} below. Recall that for a finite field $K$, an irreducible polynomial $P(X)\in K[X]$ of degree $d\geq 1$ is said to be \textit{primitive} if and only if one of (and hence all) its roots in the algebraic closure of $K$ have multiplicative order $|K|^d-1$. Fix a prime $p$. If $P_1(X),\ldots,P_r(X)$ are monic primitive irreducible polynomials over $\mathbb{F}_p$ such that, setting $d_i:=\mathrm{deg}(P_i(X))$, the tuple $(d_1,\ldots,d_r)$ is nondecreasing, we call $(d_1,\ldots,d_r)$ the \textit{Frobenius type of $\prod_{i=1}^r{\mathcal{V}(P_i(X))}$} (or of any FDG isomorphic to it). Note that elementary abelian $p$-FDGs with the same Frobenius type are isomorphic as FDSs, but not necessarily as FDGs. Indeed, it follows from uniqueness of rational canonical forms (see Subsection \ref{subsec2P4}) that if we define, for $d\in\mathbb{N}^+$, $\phi_p(d):=\frac{\phi(p^d-1)}{d}$, where $\phi$ denotes Euler's totient function, i.e., as the number of primitive irreducible polynomials of degree $d$ over $\mathbb{F}_p$ (see \cite[Theorem 3.5]{LN97a}), then the number of isomorphism types of elementary abelian $p$-FDGs $(G,\alpha)$ of Frobenius type $(d_1,\ldots,d_r)$ equals $\prod_{i=1}^r{\phi_p(d_i)}$.

\begin{corrollary}\label{strictCor}
The following is the complete (up to isomorphism) list of periodic FDGs $(G,\alpha)$ whose $\lambda$-value exceeds $\frac{1}{2}$, together with their $\lambda$-values and the information whether they are $\lambda$-maximal or not:

\begin{enumerate}
\item The FDGs \[(G,\alpha)=\prod_{i=1}^r{\mathcal{V}(P_i(X))},\] where $P_1(X),\ldots,P_r(X)$ are pairwise distinct monic primitive irreducible polynomials over $\mathbb{F}_2$ and, setting $d_i:=\mathrm{deg}(P_i(X))$ for $i=1,\ldots,r$, the numbers $d_1,\ldots,d_r$ are pairwise coprime and such that $\prod_{i=1}^r{(1-\frac{1}{2^{d_i}})}>\frac{1}{2}$. The left-hand side of that inequality equals $\lambda(\alpha)$. The underlying group of such an FDG is $(\mathbb{Z}/2\mathbb{Z})^{d_1+\cdots+d_r}$, having $\lambda$-value $1-\frac{1}{2^{d_1+\cdots+d_r}}$. Such an FDG is $\lambda$-maximal if and only if $r=1$ and $P_1(X)$ is primitive. By uniqueness of rational canonical forms, each tuple $(P_1(X),\ldots,P_r(X))$ as above yields a different FDG isomorphism type.

\item For all odd primes $p$, all $m\in\mathbb{N}^+$ and all primitive roots $a$ modulo $p^m$: \[(G,\alpha)=\mathcal{M}(p^m,a).\] The underlying group is $\mathbb{Z}/p^m\mathbb{Z}$, with $\lambda$-value $1-\frac{1}{p}$, and all these FDGs are $\lambda$-maximal.

\item For all odd primes $p$, all $m\in\mathbb{N}^+$ and all monic primitive irreducible polynomials $P(X)$ of degree $m$ over $\mathbb{F}_p$: \[(G,\alpha)=\mathcal{V}(P(X)).\] The underlying group is $(\mathbb{Z}/p\mathbb{Z})^m$, having $\lambda$-value $1-\frac{1}{p^m}$, and all these FDGs are $\lambda$-maximal. By uniqueness of rational canonical forms, each $P(X)$ yields a different FDG isomorphism type.

\item For all odd primes $p$ and all generators $g$ of $\mathbb{F}_p^{\ast}$: \[(G,\alpha)=\mathcal{V}((X-g)^2),\] with $\lambda$-value $1-\frac{1}{p}$. The underlying group is $(\mathbb{Z}/p\mathbb{Z})^2$, and these FDGs are not $\lambda$-maximal.

\item $(G,\alpha)$ which are products of one FDG of type (1), say of Frobenius type $(d_1,\ldots,d_r)$, with one FDG of type (2), (3) or (4) such that the $\Lambda$-values of the two factors are coprime and the product of the two $\lambda$-values is greater than $\frac{1}{2}$; that product value then also is the $\lambda$-value of the FDG product. As for $\lambda$-maximality: If $G_p$ denotes the Sylow $p$-subgroup of $G$ for the unique odd prime divisor $p$ of $|G|$, set \[\mathrm{L}(G_p):=\begin{cases}\{p^{m-1}(p-1)\}, & \text{ if }G_p\cong\mathbb{Z}/p^m\mathbb{Z}\text{ with }m\geq 1, \\ \{p^m-1\}, & \text{ if }G_p\cong(\mathbb{Z}/p\mathbb{Z})^m\text{ with }m>2, \\ \{p^2-1,p^2-p\}, & \text{ if }G_p\cong(\mathbb{Z}/p\mathbb{Z})^2.\end{cases}\] Then $(G,\alpha)$ is $\lambda$-maximal if and only if there exists no other additive decomposition $d'_1+\cdots+d'_s$ of $d_1+\cdots+d_r$ and no element $L\in\mathrm{L}(G_p)$ such that the $d'_j$ are pairwise coprime, each $2^{d'_j}-1$ is coprime with $L$ and $\prod_{j=1}^s{(2^{d'_j}-1)}\cdot L>\prod_{i=1}^r{(2^{d_i}-1)}\cdot\Lambda(\alpha_p)$.
\end{enumerate}
\end{corrollary}

Note that the \textit{trivial FDG} $(\{1\},\mathrm{id})$ is obtained by setting $r:=0$ in point 1 of the list in Corollary \ref{strictCor}. For a positive integer $n$, set $\mathrm{Z}_n:=\{0,\ldots,n-1\}$. Similarly to the results on $\mathrm{l}_{\alpha}$ cited above, we can also give a classification for the boundary case $\lambda(\alpha)=\frac{1}{2}$:

\begin{theoremm}\label{equalTheo}
The following is the complete (up to isomorphism) list of periodic FDGs whose $\lambda$-value equals $\frac{1}{2}$, together with the information whether they are $\lambda$-maximal or not:

\begin{enumerate}
\item $\mathbb{Z}/2\mathbb{Z}$ with the identity. It is $\lambda$-maximal.

\item $\mathbb{Z}/4\mathbb{Z}$ with multiplication by $3$. It is $\lambda$-maximal.

\item $\mathbb{Z}/2\mathbb{Z}\times\mathbb{Z}/4\mathbb{Z}=\langle v_1,v_2\mid v_1^2=v_2^4=1,[v_1,v_2]=1\rangle$ with the automorphism sending $v_1\mapsto v_1v_2^2$ and $v_2\mapsto v_1v_2$. It is $\lambda$-maximal.

\item $(\mathbb{Z}/2\mathbb{Z})^2$ with the automorphism given by the companion matrix of $(X-1)^2\in\mathbb{F}_2[X]$. It is not $\lambda$-maximal.

\item $(\mathbb{Z}/2\mathbb{Z})^3$ with the automorphism given by the companion matrix of $(X-1)^3\in\mathbb{F}_2[X]$. It is not $\lambda$-maximal.

\item $(\mathbb{Z}/2\mathbb{Z})^2\times\mathbb{Z}/3\mathbb{Z}$ with the product of the automorphism of $(\mathbb{Z}/2\mathbb{Z})^2$ given by the companion matrix of the monic primitive irreducible polynomial $X^2+X+1\in\mathbb{F}_2[X]$ with multiplication by $2$ on $\mathbb{Z}/3\mathbb{Z}$. It is $\lambda$-maximal.

\item For $n\geq 3$ and $m\in \mathrm{Z}_n$ such that $m\equiv 1\hspace{3pt}(\mathrm{mod}\hspace{3pt}p)$ for every prime divisor $p$ of $n$ and $m\equiv 1\hspace{3pt}(\mathrm{mod}\hspace{3pt}4)$ if $4\mid n$: $\mathrm{D}_{2n}=\langle r,x\mid r^n=x^2=1,xrx^{-1}=r^{-1}\rangle$ with the automorphism sending $r\mapsto r^m,x\mapsto xr$. They are all $\lambda$-maximal, and for each $n$, any two such FDGs arising from different choices of $m$ are nonisomorphic.

\item For even $n\geq 4$ and $m\in \mathrm{Z}_n$ such that $m\equiv 1\hspace{3pt}(\mathrm{mod}\hspace{3pt}p)$ for every prime divisor $p$ of $n$ and $m\equiv 1\hspace{3pt}(\mathrm{mod}\hspace{3pt}4)$ if $4\mid n$: $\mathrm{Dic}_{2n}=\langle r,x\mid r^n=1,x^2=r^{\frac{n}{2}},xrx^{-1}=r^{-1}\rangle$ with the automorphism sending $r\mapsto r^m,x\mapsto xr$. They are all $\lambda$-maximal, and for each $n$, any two such FDGs arising from different choices of $m$ are nonisomorphic.

\item For odd $o\geq 3$ and $m\in \mathrm{Z}_o$ such that $m\equiv 1\hspace{3pt}(\mathrm{mod}\hspace{3pt}p)$ for every prime divisor $p$ of $o$: $\mathrm{D}((\mathbb{Z}/2\mathbb{Z})^2\times\mathbb{Z}/o\mathbb{Z})=\langle r_1,r_2,r,x\mid r_1^2=r_2^2=r^o=[r_1,r_2]=[r_1,r]=[r_2,r]=1,x^2=1,[r_1,x]=[r_2,x]=1,xrx^{-1}=r^{-1}\rangle$ with the automorphism sending $r_1\mapsto r_2,r_2\mapsto r_1,r\mapsto r^m,x\mapsto xr_1r$. They are all $\lambda$-maximal, and for each $o$, any two such FDGs arising from different choices of $m$ are nonisomorphic.

\item For odd $o\geq 3$ and $m\in \mathrm{Z}_o$ such that $m\equiv 1\hspace{3pt}(\mathrm{mod}\hspace{3pt}p)$ for every prime divisor $p$ of $o$: $\mathrm{Dic}((\mathbb{Z}/2\mathbb{Z})^2\times\mathbb{Z}/o\mathbb{Z})=\langle r_1,r_2,r,x\mid r_1^2=r_2^2=r^o=[r_1,r_2]=[r_1,r]=[r_2,r]=1,x^2=r_1r_2,[r_1,x]=[r_2,x]=1,xrx^{-1}=r^{-1}\rangle$ with the automorphism sending $r_1\mapsto r_2,r_2\mapsto r_1,r\mapsto r^m,x\mapsto xr_1r$. They are all $\lambda$-maximal, and for each $o$, any two such FDGs arising from different choices of $m$ are nonisomorphic.
\end{enumerate}
\end{theoremm}

Note that in point 10 of Corollary \ref{strictCor}, we are committing a slight abuse of notation; for defining a generalized dicyclic group $G$ over an abelian group $A$, one needs to specify an element $y\in A$ of order $2$ to become the square of any element from $G\setminus A$. However, if $A=(\mathbb{Z}/2\mathbb{Z})^2\times\mathbb{Z}/o\mathbb{Z}$ as above, then the generalized dicyclic groups over $A$ associated with the three possible choices of $y$ are isomorphic. We therefore suppress $y$ in our notation.

One application of Corollary \ref{strictCor} is to prove that a given rational number $\rho=\frac{a}{b}\in\left(\frac{1}{2},1\right)$ with $a,b\in\mathbb{N}^+$ and $\mathrm{gcd}(a,b)=1$ is \textit{not} the $\lambda$-value of any automorphism of any finite group. For instance, we will see that this is the case if $b$ is divisible by more than one odd prime (see also Subsection \ref{subsec2P8}), so for example, no automorphism $\alpha$ of a finite group can have $\lambda$-value $\frac{8}{15}$. However, after excluding all such fractions, the remaining set of potential $\lambda$-values of periodic FDGs is still dense in $\left(\frac{1}{2},1\right)$. We can show that this is not the case for the actual image of $\lambda$, though, see points (2) and (3) in the following theorem:

\begin{theoremm}\label{denseTheo}
Consider the image $\mathrm{im}(\lambda)\subseteq\left(0,1\right]$ of the function $\lambda$ defined on the class of periodic FDGs, and set \[\rho_0:=\frac{4}{5}\cdot\prod_p{(1-\frac{1}{2^p})}=0.504307524\ldots\] and \[\rho_1:=\frac{8}{9}\cdot\prod_{n=3,5,8,49}{(1-\frac{1}{2^n})}\cdot\prod_{p\geq 11}{(1-\frac{1}{2^p})}=0.750063685\ldots,\] where in both cases, the index $p$ runs over primes. Then the following hold:

(1) For all $\rho\in\mathrm{im}(\lambda)$, there exists a sequence $((G_n,\alpha_n))_{n\geq 0}$ of periodic FDGs such that $\lambda(G_n)<\rho$ for all $n\in\mathbb{N}$ and $\mathrm{lim}_{n\to\infty}{\lambda(\alpha_n)}=\rho$. In particular, $\mathrm{im}(\lambda)$ has no isolated points.

(2) $\mathrm{im}(\lambda)\cap\left(\frac{1}{2},\rho_0\right]=\emptyset$ (in particular, $\mathrm{im}(\lambda)$ is not dense in $\left(0,1\right]$), and $\mathrm{inf}(\mathrm{im}(\lambda)\cap\left(\rho_0,1\right))=\rho_0$ (in particular, $\mathrm{im}(\lambda)$ is not closed in $\left(0,1\right]$).

(3) $\mathrm{im}(\lambda)\cap\left(\frac{3}{4},\rho_1\right]=\emptyset$ and $\mathrm{inf}(\mathrm{im}(\lambda)\cap\left(\rho_1,1\right))=\rho_1$
\end{theoremm}

This shows some differences between the functions $\lambda$ and $\mathrm{l}_{-1}$ defined on automorphisms of finite groups: $\mathrm{im}(\mathrm{l}_{-1})$ has no limit points in $\left(\frac{1}{2},1\right)$ by \cite[Corollary 3.10 and Theorem 4.2]{LM72a}, and $\frac{1}{2}$ is a limit point of $\mathrm{im}(\mathrm{l}_{-1})\cap\left(\frac{1}{2},1\right)$ by \cite[Theorem 4.13]{LM72a}. Both $\lambda$ and $\mathrm{l}_{-1}$ admit a gap above $\frac{3}{4}$, but the one of $\lambda$ is smaller. Our results on $\mathrm{im}(\lambda)\cap\left[\frac{1}{2},1\right]$ and on the structure of the underlying FDGs and groups are summarized in the following picture of the unit interval (note that the two gaps from Theorem \ref{denseTheo} are not drawn to scale):

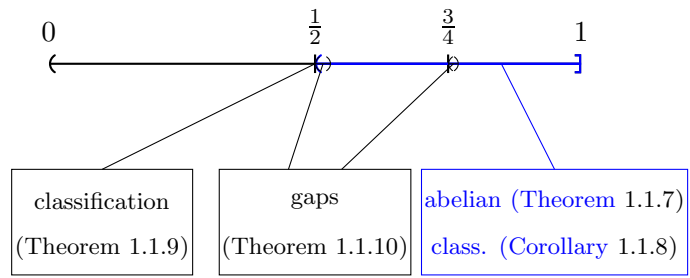
\begin{center}
\begin{tikzpicture}[scale=7]

\draw[{(-]},thick] (0,0) -- (1,0);
\foreach \x/\xtext in {0.5/$\frac{1}{2}$,0.75/$\frac{3}{4}$}
	\draw[thick] (\x,-0.5pt) -- (\x,0.5pt) node[above] {\xtext};
\node (a) at (0,0.06) {$0$};
\node (b) at (1,0.06) {$1$};
\draw[(-,thick,blue] (0.5,0) -- (1,0);
\draw[{-]},thick,blue] (0.5,0) -- (1,0);
\draw[(-)] (0.5,0) -- (0.53,0);
\draw[(-)] (0.75,0) -- (0.77,0);
\draw (0.5,0) -- (0.1,-0.2);
\draw (-0.07,-0.4) rectangle (0.27,-0.2);
\node[font=\footnotesize] (a) at (0.10,-0.26) {classification};
\node[font=\footnotesize] (b) at (0.10,-0.35) {(Theorem \ref{equalTheo})};
\draw (0.515,0) -- (0.45,-0.2);
\draw (0.76,0) -- (0.55,-0.2);
\draw (0.32,-0.4) rectangle (0.68,-0.2);
\node[font=\footnotesize] (c) at (0.5,-0.26) {gaps};
\node[font=\footnotesize] (d) at (0.5,-0.35) {(Theorem \ref{denseTheo})};
\draw[blue] (0.85,0) -- (0.95,-0.2);
\draw[blue] (0.7,-0.4) rectangle (1.2,-0.2);
\node[text=blue,font=\footnotesize] (e) at (0.95,-0.26) {abelian (Theorem \ref{abelianTheo})};
\node[text=blue,font=\footnotesize] (f) at (0.95,-0.35) {class. (Corollary \ref{strictCor})};

\end{tikzpicture}
\end{center}

\subsection{Outline of the paper}\label{subsec1P2}

Before we can prove the main results, we will need several tools, all developed in Section \ref{sec2}. Section \ref{sec3} then discusses the proofs of the main results based on the techniques developed in Section \ref{sec2}. Section \ref{sec4} provides some concluding remarks, including the aforementioned crosslink to pseudorandom number generation. We now give a sketch of the main ideas to the end of providing an overview on the paper.

\begin{enumerate}

\item One of the basic and crucial ideas is the following (for details, see the proof of Inversion Lemma \ref{inversionLem}, which can be seen as a generalization): Let $\alpha$ be an automorphism of a finite group $G$ such that $\lambda(\alpha)\geq\frac{1}{2}$, and assume, for the sake of simplicity, that $|G|\geq 3$. Then $\alpha$ has a unique cycle $\sigma$ of length $\lambda(\alpha)\cdot|G|$, consisting of nontrivial elements of $G$, and since it is so large, it must intersect with its pointwise inverse. It is not difficult to infer from this that some power of $\alpha$ inverts all elements moved by $\sigma$, and since it also inverts $1_G$, $G$ has an automorphism inverting more than half of its elements.

\item Finite groups with automorphisms inverting more than half of their elements have been extensively studied by Liebeck and MacHale in \cite{LM72a}. Of course, all finite abelian groups are such groups, and for the nonabelian ones, Liebeck and MacHale showed a structure theorem \cite[Theorem 4.13]{LM72a}, which we will later recall, for the readers' convenience, together with some of their other results, in Subsection \ref{subsec2P2}. For us, the most important conclusion to draw from the structure theorem is that for any finite nonabelian group $G$ having an automorphism inverting more than half of the elements of $G$, either $G$ is nilpotent of class $2$ and we have structural information on the central quotient of $G$, or we know that $G$ has an abelian subgroup of index $2$.

\item In Subsection \ref{subsec2P1}, we will prove several lemmata for a systematic study of LCCs in finite groups, most notably the Transfer Lemma \ref{transferLem}, by which we can transfer LCCs to characteristic quotients and at least infer \enquote{affine LCCs} (see Definition \ref{affineDef}) on characteristic subgroups from automorphism LCCs on the entire group.

\item By said transfer results, we will first gain a complete understanding of the $\frac{1}{2}$-LCC (and the boundary cases with $\lambda$-value equal to $\frac{1}{2}$ as well) in finite \textit{abelian} groups. This becomes possible through a perspective on automorphisms of finite abelian groups worked out in detail by Hillar and Rhea \cite{HR07a}, which allows us to introduce a concept of \enquote{compatibility} (for transfer of LCCs) of finite abelian $p$-groups, see Definition \ref{compatibleDef} and the subsequent Compatibility Lemma \ref{compLem}. This will reduce the problem to the study of a few special cases.

\item Most notable among the special cases to be considered is the elementary abelian case. It is well-known that automorphisms of $(\mathbb{Z}/p\mathbb{Z})^n$ can be seen as $\mathbb{F}_p$-vector space automorphisms, and so the problem of classifying the elementary abelian FDGs of $\lambda$-value at least $\frac{1}{2}$ actually is a problem from the theory of \textit{linear finite dynamical systems} (FDSs arising from the actions of endomorphisms of finite vector spaces). By an idea dating back to Elspas \cite{Els59a} (see \cite{Her05a} for a concise, modern exposition), a powerful tool for studying questions about linear finite dynamical systems are primary rational canonical forms of square matrices over finite fields, and we will follow this idea to gain a complete understanding of the $\frac{1}{2}$-LCC and its boundary cases for elementary abelian FDGs in Subsection \ref{subsec2P4}.

\item The deeper understanding of the elementary abelian case developed in Subsection \ref{subsec2P4} allows us, apart from proving classification results for finite abelian groups in general in Subsections \ref{subsec2P5} and \ref{subsec2P6}, to prove lower bounds on $\lambda$-values of automorphisms of finite elementary abelian $2$-groups (the most complicated among the finite elementary abelian groups as far as the $\frac{1}{2}$-LCC is concerned) in Subsection \ref{subsec2P7}, see Lemma \ref{anLem}.

\item For proving Theorems \ref{abelianTheo} and \ref{equalTheo}, we can then use the transfer results from Subsection \ref{subsec2P1} to infer information on some abelian FDG associated with a fixed finite nonabelian FDG $(G,\alpha)$ such that $\lambda(\alpha)\geq\frac{1}{2}$, strengthen this information using our deeper understanding of the abelian case, and \enquote{reflect back} to $(G,\alpha)$; this will restrict the structure of $(G,\alpha)$ strongly enough to give a complete classification of all such $(G,\alpha)$ as well.

\item The proofs of the first parts of Theorem \ref{denseTheo}(2,3) are mainly an application of the \enquote{analytic} understanding of $\lambda$-values of automorphisms of finite elementary abelian $2$-groups gained in Subsection \ref{subsec2P5}, together with the general classification results proved before. We will be using some sort of \enquote{back and forth} argument, where basically, one observes that including one of the potential factors making up the $\lambda$-value would cause it to drop too much, but leaving it out, the $\lambda$-value is too large to fall into the interval in question. Of course, this is just the basic idea, and the precise argument is more involved, especially for point (3).

\end{enumerate}

Finally, a few words on algorithmic aspects and on the use of computers for this paper: As a consequence of the first three main results, one can design an algorithm which on input a rational number $\rho$ from $\left[\frac{1}{2},1\right]$ outputs a complete description of the (possibly infinite) list of periodic FDGs $(G,\alpha)$ such that $\lambda(\alpha)=\rho$. This algorithm is described in detail in Subsection \ref{subsec2P8}, where we also use some of the main ideas for the algorithm to prove the \enquote{abelian part} of Theorem \ref{equalTheo} \enquote{by hand}, see Proposition \ref{abEqProp}. In the preparation of this paper, the author has implemented said algorithm in GAP \cite{GAP4} and also written some other GAP code which can be used to make some proofs a bit shorter, for example by effective checking of boundary cases. We do \textit{not} give the explicit code in this paper, but a text file containing the complete code together with a documentation file is available from the author upon request. Note that none of these applications of algorithms is essential for our proofs, though, since the cases checked algorithmically are sufficiently simple to allow for feasibly short computations \enquote{by hand}. We also note that in Subsection \ref{subsec2P7} and in the proof of Theorem \ref{denseTheo}, some numerical bounds are used, all of which were checked with GAP \cite{GAP4}.

\subsection{Notation}

We now provide an overview on the notation used in this paper, some of which has already been used by now. By $\mathbb{N}$, we denote the set of natural numbers (including $0$), and by $\mathbb{N}^+$ the set of positive integers. The image of a set $M$ under a function $f$ is denoted by $f[M]$, the restriction of $f$ to $M$ by $f_{\mid M}$, and the domain of $f$ by $\mathrm{dom}(f)$. For any set $X$, the symmetric group on $X$ is denoted by $\mathcal{S}_X$, and for any $\psi\in\mathcal{S}_X$, there is an equivalence relation $\sim_{\psi}$ on $X$ associated with $\psi$, given by $x\sim_{\psi} y:\Leftrightarrow\exists n\in\mathbb{Z}:\psi^n(x)=y$. If $X$ is finite, we view the equivalence classes of $\sim_{\psi}$ as the domains of the disjoint cycles into which $\psi$ decomposes.

Euler's totient function is denoted by $\phi$, which is to be distinguished from the symbol $\varphi$ used for group endomorphisms. For any $n\in\mathbb{N}^+$, $Z_n$ denotes the set $\{0,\ldots,n-1\}$, a set of representatives of the cosets of $n\mathbb{Z}$ in $\mathbb{Z}$. For a prime $p$ and an integer $a$, $\nu_p(a)$ denotes the $p$-adic valuation of $a$, i.e., the exponent of $p$ in the prime power factorization of $a$. Using $p$ as the running index in a product always implies that it runs over primes, subject to the other conditions specified.

For a primary number (i.e., prime power) $q$, the finite field with $q$ elements is denoted by $\mathbb{F}_q$. For a commutative ring $R$ and $n\in\mathbb{N}^+$, we denote by $\mathrm{Mat}_n(R)$ the ring of $(n\times n)$-matrices over $R$, and by $\mathrm{GL}_n(R)$ the group of regular $(n\times n)$-matrices over $R$. For a field $K$, we denote the algebraic closure of $K$ by $\overline{K}$ and the group of units of $K$ by $K^{\ast}$.

The (natural) exponential function $\mathbb{R}\rightarrow\mathbb{R}$ is denoted by $\mathrm{exp}$. We use the symbol \qedsymbol\hspace{0.5pt} either to signal the end of a proof, or, when placed directly at the end of a theorem, lemma, etc., to indicate that the result is not proved in this paper (in which case a reference is given).

\section{The tools for proving the main results}\label{sec2}

\subsection{Some useful lemmata}\label{subsec2P1}

We introduce the following concepts, the first of which is a generalization of the notion of a group endomorphism:

\begin{deffinition}\label{affineDef}
(1) Let $G$ be a group, $g_0\in G$ and $\varphi$ an endomorphism of $G$. We call the map $G\rightarrow G$, sending $g\mapsto g_0\varphi(g)$ for all $g\in G$, the \textbf{(left-)affine map of $G$ w.r.t.~$g_0$ and $\varphi$}, denoted by $\mathrm{A}_{g_0,\varphi}$.

(2) If $G$ is a finite group and $A$ is an affine map of $G$, the pair $(G,A)$ is called a \textbf{generalized FDG} (\textbf{gFDG}). It is called \textbf{abelian} if and only if $G$ is abelian.

(3) We call an affine map $\mathrm{A}_{g_0,\varphi}$ (or the associated gFDG) \textbf{periodic} if and only if $\mathrm{A}_{g_0,\varphi}$ is a permutation of $G$, which is easily seen to hold if and only if $\varphi$ is an automorphism of $G$, and denote the set of all periodic affine maps of $G$ by $\mathrm{Aff}(G)$.

(4) For a finite group $G$, we define $\Lambda_{\mathrm{aff}}(G):=\mathrm{max}_{A\in\mathrm{Aff}(G)}{\Lambda(A)}$ and $\lambda_{\mathrm{aff}}(G):=\frac{1}{|G|}{\Lambda_{\mathrm{aff}}(G)}$.

(5) For a fixed $\rho\in\left(0,1\right)$, we say that a finite group $G$ satisfies the \textbf{affine $\rho$-LCC} if and only if $\lambda_{\mathrm{aff}}(G)>\rho$.
\end{deffinition}

As all endomorphisms of a group are affine maps, the class of gFDGs extends the class of FDGs. Viewing gFDGs as algebraic structures just as we did for FDGs before, we find that an \textit{isomorphism between gFDGs $(G_1,A_1)$ and $(G_2,A_2)$} is a group isomorphism $f:G_1\rightarrow G_2$ such that $f\circ A_1=A_2\circ f$. Clearly, $\lambda_{\mathrm{aff}}(G)\geq\lambda(G)$, whence the $\rho$-LCC implies the affine $\rho$-LCC. Denote by $\mu:G\rightarrow\mathcal{S}_G$ the left regular representation of $G$. Then by definition, $\mathrm{Aff}(G)=\mu[G]\mathrm{Aut}(G)$ in $\mathcal{S}_G$. It is well-known (see, for example, \cite[p.~37]{Rob96a}) and follows easily from the formula $\mathrm{A}_{g_1,\alpha_1}\circ\mathrm{A}_{g_2,\alpha_2}=\mathrm{A}_{g_1\alpha_1(g_2),\alpha_1\circ\alpha_2}$ that the subgroup of $\mathcal{S}_G$ generated by $\mu[G]$ and $\mathrm{Aut}(G)$ is their internal semidirect product, which is canonically (via the map $\mu(g)\alpha\mapsto (g,\alpha)$) isomorphic to the holomorph of $G$, $\mathrm{Hol}(G)$, defined as the external semidirect product $G\rtimes\mathrm{Aut}(G)$ with respect to the natural action of $\mathrm{Aut}(G)$ on $G$. In particular, the periodic (left-)affine maps of $G$ form a subgroup of $\mathcal{S}_G$. One could similarly define the concept of a \textit{right-affine map of $G$} as a map sending $g\mapsto \varphi(g)g_0$ for all $g\in G$ for some fixed endomorphism $\varphi$ of $G$ and element $g_0\in G$, and observe that the periodic (i.e., bijective) right-affine maps of $G$ also form a subgroup of $\mathcal{S}_G$, which is similar to $\mathrm{Aff}(G)$ (more precisely, each of the two subgroups is the image of the other under conjugation by the inversion on $G$). Therefore, restricting our attention to left-affine maps does not result in a loss of generality of the theory.

Futhermore, the isomorphism between $\mathrm{Aff}(G)$ and $\mathrm{Hol}(G)$ discussed above yields useful upper bounds on orders of affine maps of finite groups in terms of maximum element and maximum automorphism orders. For a more concise formulation, we make the following definition (the first half of which is from \cite{GMPS15a}):

\begin{deffinition}\label{meoDef}
Let $G$ be a finite group. We define $\mathrm{meo}(G)$ to be the maximum element order of $G$ and set $\mathrm{mao}(G):=\mathrm{meo}(\mathrm{Aut}(G))$, the maximum automorphism order of $G$.
\end{deffinition}

\begin{lemmma}\label{affineLem}
(1) Let $G$ be a finite group, $g_0\in G$ and $\alpha$ an automorphism of $G$. Set $f:=g_0\alpha(g_0)\cdots\alpha^{\mathrm{ord}(\alpha)-1}(g_0)$. Then $\mathrm{ord}(\mathrm{A}_{g_0,\alpha})=\mathrm{ord}(f)\cdot\mathrm{ord}(\alpha)$. In particular, $\mathrm{ord}(\mathrm{A}_{g_0,\alpha})\leq\mathrm{meo}(G)\cdot\mathrm{ord}(\alpha)$, and if $G$ is abelian, then $\mathrm{ord}(\mathrm{A}_{g_0,\alpha})\leq \mathrm{meo}(\mathrm{fix}(\alpha))\cdot\mathrm{ord}(\alpha)$, where $\mathrm{fix}(\alpha)$ is the subgroup of $G$ consisting of fixed points of $\alpha$.

(2) For any finite group $G$, we have that $\mathrm{meo}(\mathrm{Aff}(G))\leq\mathrm{meo}(G)\cdot\mathrm{mao}(G)$.
\end{lemmma}

\begin{proof}
For (1): Note that by the isomorphism mentioned above, the order of $\mathrm{A}_{g_0,\alpha}$ in $\mathrm{Aff}(G)$ equals the order of $(g_0,\alpha)$ in $\mathrm{Hol}(G)$, which is clearly divisible by $\mathrm{ord}(\alpha)$, and since by definition, $(g_0,\alpha)^{\mathrm{ord}(\alpha)}=(f,\mathrm{id})$, the equation for $\mathrm{ord}(\mathrm{A}_{g_0,\alpha})$ follows. For the \enquote{In particular}: The first factor in the product on the RHS of the equation is always bounded above by $\mathrm{meo}(G)$, and an easy computation reveals that $\alpha(f)=g_0^{-1}fg_0$, whence $f$ is a fixed point of $\alpha$ if $G$ is abelian.

For (2): This follows immediately from (1).
\end{proof}

We will now begin to prove the announced transfer results. Recall that for a group $G$, a subgroup $H$ and an automorphism $\alpha$ of $G$, $H$ is called \textit{$\alpha$-admissible} if and only if $\alpha[H]=H$ (that is, $\alpha$ restricts to an automorphism of $H$). Also, recall that for any group $G$, any normal subgroup $N\unlhd G$ and any automorphism $\alpha$ of $G$, the following are equivalent:

(1) $N$ is $\alpha$-admissible.

(2) There exists an automorphism $\tilde{\alpha}$ of $G/N$ making the following diagram commute (where $\pi:G\rightarrow G/N$ is the canonical projection):

\begin{center}
\begin{tikzpicture}
\matrix (m) [matrix of math nodes, row sep=3em,
column sep=3em]
{ G & G \\
G/N & G/N \\
};
\path[->]
(m-1-1) edge node[above] {$\alpha$} (m-1-2)
(m-1-1) edge node[left] {$\pi$} (m-2-1)
(m-1-2) edge node[right] {$\pi$} (m-2-2)
(m-2-1) edge node[below] {$\tilde{\alpha}$} (m-2-2);
\end{tikzpicture}
\end{center}

$\tilde{\alpha}$ is unique and is called the \textit{automorphism of $G/N$ induced by $\alpha$}. More generally, when we have a commutative diagram as above, but where $\alpha$ is not necessarily an automorphism, but just some permutation of $G$, then we still call $\tilde{\alpha}$ \textit{induced by $\alpha$}, and for each cycle $\sigma$ of $\alpha$, the image of $\mathrm{dom}(\sigma)$ under $\pi$ then is the domain of a unique cycle $\tilde{\sigma}$ of $\tilde{\alpha}$, which we call the \textit{cycle of $\tilde{\alpha}$ induced by $\sigma$}.

\begin{lemmma}\label{transferLem}(\enquote{Transfer Lemma})
Let $G$ be a finite group, $g_0\in G$, $\alpha$ an automorphism of $G$, $H$ an $\alpha$-admissible subgroup, $N$ an $\alpha$-admissible normal subgroup of $G$, and denote by $\tilde{\alpha}$ the automorphism of $G/N$ induced by $\alpha$ and by $\pi:G\rightarrow G/N$ the canonical projection. Set $A:=\mathrm{A}_{g_0,\alpha}$ and $\tilde{A}:=\mathrm{A}_{\pi(g_0),\tilde{\alpha}}$. Then the following hold:

(1) If $xH$ is a left coset of $H$ in $G$ such that $A[xH]=xH$, say $A(x)=xh_0$, then the FDSs $(xH,A_{\mid xH})$ and $(H,\mathrm{A}_{h_0,\alpha_{\mid H}})$ are isomorphic via the map $\tau:H\rightarrow xH,h\mapsto xh$.

(2) $\tilde{A}$ is induced by $A$. Let $x\in G$ be such that the length of the cycle $\sigma$ of $x$ under $A$ equals $\Lambda(A)$, and let $L\in\mathbb{N}^+$ be minimal such that $A^L(x)\in xN$, say $A^L(x)=xn_0$. Then $L$ is the length of the cycle $\tilde{\sigma}$ of $\tilde{A}$ induced by $\sigma$, $L\mid\Lambda(A)$, and its complementary divisor $l$ is the length of the cycle of $1_N=1_G$ under the affine map $\mathrm{A}_{n_0,\alpha^L_{\mid N}}$ of $N$. In particular, we have that \[\lambda(A)=\frac{L}{|G/N|}\cdot\frac{l}{|N|}\leq\lambda(\tilde{A})\cdot\lambda(\mathrm{A}_{n_0,\alpha^L_{\mid N}}).\]

(3) For all finite groups $G$ and all characteristic subgroups $N$ of $G$, we have the following:

(a) $\lambda(G/N)\geq\lambda(G)$ (\enquote{Automorphism Quotient Transfer}),

(b) $\lambda_{\mathrm{aff}}(G/N)\geq\lambda_{\mathrm{aff}}(G)$ (\enquote{Affine Quotient Transfer}) and

(c) $\lambda_{\mathrm{aff}}(N)\geq\lambda_{\mathrm{aff}}(G)$ (\enquote{Affine Subgroup Transfer}).
\end{lemmma}

\begin{proof}
For (1): By assumption, we have $xh_0=\mathrm{A}_{g_0,\alpha}(x)=g_0\alpha(x)$. Using this, we find that indeed, for all $h\in H$, we have $\mathrm{A}_{g_0,\alpha}(\tau(h))=g_0\alpha(xh)=g_0\alpha(x)\alpha(h)=xh_0\alpha(h)=\tau(\mathrm{A}_{h_0,\alpha_{\mid H}}(h))$.

For (2): It is readily verified that $\tilde{A}$ is induced by $A$, and it is clear that $L$ is the length of $\tilde{\sigma}$. In particular, $L\mid\Lambda(A)$, $B:=A^L$ restricts to a permutation of $xN$, and $l$ is just the length of the cycle of $x$ under $B$, which under the isomorphism from (1) corresponds to the cycle of $1_N$ under $\mathrm{A}_{n_0,\alpha^L_{\mid N}}$. In particular, $\frac{L}{|G/N|}$ is the share of the length of some cycle of $\tilde{A}$ in $G/N$, thus bounded above by $\lambda(\tilde{A})$, and $\frac{l}{|N|}$ is some cycle length share of the affine map $\mathrm{A}_{n_0,\alpha^L_{\mid N}}$ of $N$, whence it is bounded above by $\lambda(\mathrm{A}_{n_0,\alpha^L_{\mid N}})$.

For (3): All three points follow from (2). (a) follows by choosing $A$ as an automorphism $\alpha$ of $G$ such that $\lambda(\alpha)=\lambda(G)$; then $\tilde{A}=\tilde{\alpha}$ is an automorphism of $G/N$, whence we can bound the first factor from the bound in (2) by $\lambda(G/N)$, and the second by $1$. For (b) and (c), choose $A$ as a periodic affine map of $G$ such that $\lambda(A)=\lambda_{\mathrm{aff}}(G)$; for (b), bound the first factor by $\lambda_{\mathrm{aff}}(G/N)$ and the second by $1$, and for (c), bound the first factor by $1$ and the second by $\lambda_{\mathrm{aff}}(N)$.
\end{proof}

\begin{remmark}\label{transferRem}
(1) Note that there is no \enquote{Automorphism Subgroup Transfer}, i.e., in general, we \textit{cannot} conclude $\lambda(N)\geq\lambda(G)$ for characteristic subgroups $N$ of finite groups $G$. Indeed, it is not difficult to see that $\lambda(\mathbb{Z}/6\mathbb{Z})=\frac{1}{3}$, but $\mathrm{D}_{12}$ has a characteristic subgroup isomorphic to $\mathbb{Z}/6\mathbb{Z}$, and $\lambda(\mathrm{D}_{12})=\frac{1}{2}$.

For the rest of this remark, we make the following definitions, generalizing the respective notions for finite groups: For every finite semigroup $S$ and all $s_0\in S$, $\alpha\in\mathrm{Aut}(S)$:

(i) $\lambda(S):=\mathrm{max}_{\beta\in\mathrm{Aut}(S)}{\lambda(\beta)}$,

(ii) $\mathrm{A}_{s_0,\alpha}:S\rightarrow S,s\mapsto s_0\alpha(s)$, the \textit{(left-)affine map of $S$ with respect to $s_0$ and $\alpha$},

(iii)  $\mathrm{Aff}(S):=(\{\mathrm{A}_{s,\beta}\mid s\in S,\beta\in\mathrm{Aut}(S)\}\cap\mathcal{S}_S)\cup\mathrm{Aut}(S)$, the set of \textit{periodic affine maps of $S$} (which may be empty, and contains $\mathrm{Aut}(S)$ if $S$ is a monoid), and

(iv) if $\mathrm{Aff}(S)\not=\emptyset$: $\lambda_{\mathrm{aff}}(S):=\mathrm{max}_{A\in\mathrm{Aff}(S)}{\lambda(A)}$.

(2) A statement analogous to the one of Lemma \ref{transferLem}(1) for finite monoids is obviously false (in monoids, $H$ and $xH$ need not even be of the same cardinality), but the following is true, with an analogous proof: Let $S$ be a finite semigroup, $A=\mathrm{A}_{s_0,\alpha}$ a periodic affine map and $T$ a subsemigroup of $S$ such that $\alpha[T]=T$. Let $C=xT$ for some $x\in S$ and assume that $A[C]=C$. Then, denoting by $\overline{A}$ the restriction of $A$ to $C$, there exists an element $t_0\in T$ such that the FDS $(C,\overline{A})$ is a homomorphic image of the FDS $(T,\mathrm{A}_{t_0,\alpha})$.

(3) The proof of Lemma \ref{transferLem}(2) makes use of the fact that for any group $G$ and any group congruence $\equiv$ on $G$, the congruence classes of $\equiv$ all have the same cardinality. This is not true for monoids, and indeed, natural generalizations of the three points in (3) are all false for finite monoids. Call a monoid congruence $\equiv$ on $M$ \textit{characteristic} if and only if every automorphism of $M$ induces (in the sense of commutativity of a diagram as above) an automorphism of $M/{\equiv}$, and call a submonoid $N$ of $M$ \textit{characteristic} if and only if all automorphisms of $M$ restrict to automorphisms of $N$. Then for every $1>\epsilon>0$, there exists a finite monoid $M_{\epsilon}$, a characteristic monoid congruence $\equiv$ on $M_{\epsilon}$ and a characteristic submonoid $N$ of $M_{\epsilon}$ such that $\lambda(M_{\epsilon})>1-\epsilon$ , whereas $\mathrm{max}\{\lambda_{\mathrm{aff}}(M_{\epsilon}/{\equiv}),\lambda_{\mathrm{aff}}(N)\}<\epsilon$. To see this, set $n:=\lfloor\frac{1}{\epsilon}\rfloor+1$, and choose a prime $p$ so large that $\frac{p-1}{p+n-1}>1-\epsilon$. Let $M_{\epsilon}$ be the finite monoid obtained by successive adjunction of $n-1$ zero elements $z_1,\ldots,z_{n-1}$ to $\mathbb{Z}/p\mathbb{Z}$, let $\equiv$ be the congruence on $M_{\epsilon}$ having congruence classes $\mathbb{Z}/p\mathbb{Z},\{z_1\},\ldots,\{z_{n-1}\}$, and set $N:=\{0,z_1,\ldots,z_{n-1}\}$ (where $0\in\mathbb{Z}/p\mathbb{Z}$ is the identity element). It is then not difficult to check the following:

(i) Any automorphism of $M_{\epsilon}$ restricts to a permutation on each congruence class of $\equiv$. In particular, $\equiv$ is characteristic,

(ii) $N$ is characteristic in $M_{\epsilon}$,

(iii) $\lambda(M)=\frac{p-1}{p+n-1}>1-\epsilon$,

(iv) ${(M_{\epsilon}/{\equiv})}\cong N$, and both only have one periodic affine map (namely the identity), whence $\lambda_{\mathrm{aff}}(M_{\epsilon}/{\equiv})=\lambda_{\mathrm{aff}}(N)=\frac{1}{n}<\epsilon$.
\end{remmark}

We say that a family $(G_i)_{i\in I}$ of groups has the \textit{splitting property} if and only if every automorphism $\alpha$ of $\prod_{i\in I}{G_i}$ \textit{splits}, i.e., there exists a family $(\alpha_i)_{i\in I}$ with $\alpha_i$ an automorphism of $G_i$ for all $i\in I$ such that $\alpha=\prod_{i\in I}{\alpha_i}$ in the sense of component-wise application. Since tuples of finite groups with pairwise coprime orders have the splitting property, the following lemma is particularly useful in the study of the $\frac{1}{2}$-LCC in finite abelian groups:

\begin{lemmma}\label{productLem}(\enquote{Product Lemma})
(1) Let $(S_1,\psi_1),\ldots,(S_r,\psi_r)$ be periodic FDSs. Then $\lambda(\prod_{i=1}^r{\psi_i})=\prod_{i=1}^r{\lambda(\psi_i)}$ if and only if the $\Lambda(\psi_i)$, $i=1,\ldots,r$, are pairwise coprime, and else $\lambda(\prod_{i=1}^r{\psi_i})\leq\frac{1}{2}\prod_{i=1}^r{\lambda(\psi_i)}$.

(2) Let $(G_1,\ldots,G_r)$ be a tuple of finite groups with the splitting property. Then the following hold:

(i) The periodic affine maps of $G:=\prod_{i=1}^r{G_i}$ are precisely the products $\prod_{i=1}^r{A_i}$, where $A_i$ is a periodic affine map of $G_i$ for $i=1,\ldots,r$. Furthermore, such a product is an automorphism of $G$ if and only if $A_i$ is an automorphism of $G_i$ for $i=1,\ldots,r$.

(ii) For a periodic affine map $A=\prod_{i=1}^r{A_i}$ of $G$, $\lambda(A)=\prod_{i=1}^r{\lambda(A_i)}$ if and only if the $\Lambda(A_i)$ are pairwise coprime, and else $\lambda(A)\leq\frac{1}{2}\prod_{i=1}^r{\lambda(A_i)}$. In particular, if $\alpha=\prod_{i=1}^r{\lambda(\alpha_i)}$ is an automorphism of $G$ such that the $\Lambda(\alpha_i)$, $i=1,\ldots,r$, are not pairwise coprime, then $\lambda(\alpha)<\frac{1}{2}$.

(iii) $\lambda(\prod_{i=1}^r{G_i})\leq\prod_{i=1}^r{\lambda(G_i)}$, with equality if and only if the $\Lambda(G_i)$, $i=1,\ldots,r$, are pairwise coprime. 

(iv) $\lambda_{\mathrm{aff}}(\prod_{i=1}^r{G_i})\leq\prod_{i=1}^r{\lambda_{\mathrm{aff}}(G_i)}$, with equality if and only if the $\Lambda_{\mathrm{aff}}(G_i)$, $i=1,\ldots,r$, are pairwise coprime. In particular, if the orders of the $G_i$ are pairwise coprime, then $\lambda_{\mathrm{aff}}(G)=1$ if and only if $\lambda_{\mathrm{aff}}(G_i)=1$ for $i=1,\ldots,r$.
\end{lemmma}

\begin{proof}
For (1): This follows immediately from the fact that the cycle length of a tuple $(s_1,\ldots,s_r)\in\prod_{i=1}^r{S_i}$ under $\prod_{i=1}^r{\psi_i}$ equals the least common multiple of the cycle lengths of the $s_i$ under $\psi_i$.

For (2,i): This follows immediately from the identity $\mathrm{A}_{(g_1,\ldots,g_n),\prod_{i=1}^n{\alpha_i}}=\prod_{i=1}^n{\mathrm{A}_{g_i,\alpha_i}}$ for all $(g_1,\ldots,g_n)\in\prod_{i=1}^n{G_i}$ and $(\alpha_1,\ldots,\alpha_n)\in\prod_{i=1}^n{\mathrm{Aut}(G_i)}$, which is easy to check.

For (2,ii): The first assertion follows immediately from (1), and the second assertion follows from the first by observing that $\lambda(G)<1$ for all nontrivial finite groups $G$.

For (2,iii) and (2,iv): These follow from (ii).
\end{proof}

We conclude this subsection with the following lemma:

\begin{lemmma}\label{inversionLem}(\enquote{Inversion Lemma})
Let $G$ be a finite group, $\alpha$ an automorphism of $G$ with $\lambda(\alpha)\geq\frac{1}{2}$, $A\leq\zeta G$ such that $A$ is proper in $G$. Then $A$ is $\alpha$-admissible and, denoting by $\tilde{\alpha}$ the automorphism of $G/A$ induced by $\alpha$, some power of $\alpha$ inverts more than the fraction $\lambda(\tilde{\alpha})$ of the elements in $G$.
\end{lemmma}

\begin{proof}
The assertion is vacuously true if $G$ is trivial, and it is also readily verified for $G=\mathbb{Z}/2\mathbb{Z}$, whence we may assume that $|G|\geq 3$. Then $\alpha$ has a unique largest cycle $\sigma$, and $\mathrm{dom}(\sigma)\subseteq G\setminus\{1_G\}$. Denoting by $\mathrm{dom}(\sigma)^{-1}$ the pointwise inverse of $\mathrm{dom}(\sigma)$, we find that since both $\mathrm{dom}(\sigma)$ and $\mathrm{dom}(\sigma)^{-1}$ are subsets of $G\setminus\{1_G\}$ of size more than half of the size of that set, they must intersect. That is, there exists $x\in\mathrm{dom}(\sigma)$ such that $x^{-1}\in\mathrm{dom}(\sigma)$ as well. In particular, there exists $k\in\mathbb{N}$ such that $\alpha^k(x)=x^{-1}$. Now let $y\in\mathrm{dom}(\sigma)$ be arbitrary, but fixed. Then there exists $n\in\mathbb{N}$ such that $y=\alpha^n(x)$, and hence \[\alpha^k(y)=\alpha^k(\alpha^n(x))=\alpha^n(\alpha^k(x))=\alpha^n(x^{-1})=\alpha^n(x)^{-1}=y^{-1},\] so $\alpha^k$ inverts all elements from $\mathrm{dom}(\sigma)$.

Now let $A'\subseteq A$ be such that $\mathrm{dom}(\sigma)\cap xA=xA'$. By Transfer Lemma \ref{transferLem}(2) and nontriviality of $G/A$, we find that \[\frac{|A'|}{|A|}\geq\frac{\lambda(\alpha)}{\lambda(\tilde{\alpha})}>\lambda(\alpha)\geq\frac{1}{2},\] and hence $|A'|>\frac{1}{2}|A|$. From the fact that $\alpha^k$ inverts all elements of $xA'$ and that $1_G\in A'\subseteq\zeta G$, it follows immediately that $\alpha^k$ inverts all elements of $A'$. In particular, $\alpha^k$ inverts more than half of the elements of $A$, so $\alpha^k[A]=A$ by Lagrange's theorem. Furthermore, $A$ is abelian and in an abelian group, the subset of elements inverted by a fixed automorphism forms a subgroup, whence $\alpha^k$ inverts all elements of $A$ and thus all elements of $xA$, hence all elements from the preimage under the canonical projection $G\rightarrow G/A$ of the domain of the cycle $\tilde{\sigma}$ of $\tilde{\alpha}$ induced by $\sigma$, giving a fraction of $\lambda(\tilde{\alpha})$ elements of $G$ inverted by $\alpha$. Since $\mathrm{dom}(\sigma)\cap A=\emptyset$ by the size of $\mathrm{dom}(\sigma)$, $1_G$ is not among the elements inverted by $\alpha^k$ which we already considered, and so $\alpha^k$ indeed inverts more than the fraction $\lambda(\tilde{\alpha})$ of the elements in $G$.
\end{proof}

\subsection{Groups with automorphisms inverting more than half of their elements}\label{subsec2P2}

As observed in Subsection \ref{subsec1P2}, Inversion Lemma \ref{inversionLem} is a generalization of the observation from point 1 in Subsection \ref{subsec1P2}, and indeed, setting $A:=\{1_G\}$ in Inversion Lemma \ref{inversionLem}, we find that every finite group $G$ with $\lambda(G)\geq\frac{1}{2}$ has an automorphism inverting more than half of its elements. In order to be able to refer to them later, we will now quickly present those of the aforementioned results of Liebeck and MacHale \cite{LM72a} which are of relevance for our paper. For the sake of simplicity, we replace the notation $\mathrm{l}_{-1}$ by $\mathrm{l}$:

\begin{deffinition}\label{lDef}
Let $G$ be a finite group. For an automorphism $\alpha$ of $G$, we set $\mathrm{l}(\alpha):=\frac{1}{|G|}|\{g\in G\mid \alpha(g)=g^{-1}\}|$ and $\mathrm{l}(G):=\mathrm{max}_{\alpha\in\mathrm{Aut}(G)}{\mathrm{l}(\alpha)}$.
\end{deffinition}

\begin{theoremm}\label{lTheo}(\cite[Corollary 3.10 and Theorem 4.2]{LM72a})
Let $G$ be a finite group such that $\mathrm{l}(G)>\frac{1}{2}$. Then for all automorphisms $\alpha$ of $G$ such that $\mathrm{l}(\alpha)>\frac{1}{2}$, we have $\mathrm{l}(\alpha)=\mathrm{l}(G)$. Furthermore, there exists $k\in\mathbb{N}^+$ such that $\mathrm{l}(\alpha)=\frac{k+1}{2k}$.\qed
\end{theoremm}

\begin{corrollary}\label{millerCor}(folklore due to Miller \cite{Mil29a})
A finite group $G$ with $\mathrm{l}(G)>\frac{3}{4}$ is abelian.\qed
\end{corrollary}

\begin{theoremm}\label{lStructureTheo}(\cite[Theorem 4.13]{LM72a})
Let $G$ be a finite nonabelian group such that $\mathrm{l}(G)>\frac{1}{2}$. Then $G$ is of one of the following three types:

(I) $G$ has an abelian subgroup $A$ of index $2$. Conversely, for any nonabelian such $G$, $\mathrm{l}(G)=\frac{q+1}{2q}$, where $q=[A:\mathrm{C}_A(x)]$ for any $x\in G\setminus A$.

(II) $G$ is nilpotent of class $2$, with cyclic commutator subgroup of order $2$, say generated by $z$. Also, its center has index $2^{2k}$ in $G$ for some $k\geq 2$, with elementary abelian central quotient, generated by the images of elements $x_1,\ldots,x_k,a_1,\ldots,a_k\in G$ subject, in $G$, to the following commutator relations for all $i,j,l\in\{1,\ldots,k\}$ with $i\not=l$: $[x_i,x_j]=[a_i,a_j]=1,[a_i,x_l]=1,[a_i,x_i]=z$. Conversely, any such $G$ satisfies $\mathrm{l}(G)=\frac{2^k+1}{2^{k+1}}$.

(III) $G$ is nilpotent of class $2$, with elementary abelian commutator subgroup of order $4$, say generated by $z_1$ and $z_2$. Also, its center has index $2^4$ in $G$, and the central quotient is elementary abelian, generated by the images of appropriate elements $x_1,x_2,a_1,a_2$ subject, in $G$, to the following commutator relations: $[x_1,x_2]=[a_1,a_2]=[a_1,x_2]=[a_2,x_1]=1,[a_1,x_1]=z_1,[a_2,x_2]=z_2$. Conversely, any such $G$ satisfies $\mathrm{l}(G)=\frac{9}{16}$.\qed
\end{theoremm}

\subsection{Automorphisms of finite abelian groups and \enquote{compatibility}}\label{subsec2P3}

This subsection is inspired by Hillar and Rhea's paper \cite{HR07a}. Let $H$ be a finite abelian group. Then, as any finite nilpotent group, $H$ is the direct product of its Sylow subgroups: $H=\prod_p{H_p}$, and the family $(H_p)_p$ has the splitting property; if $\alpha=\prod_p{\alpha_p}$ for some automorphism $\alpha$ of $H$, we call $(H_p,\alpha_p)$ the \textit{Sylow $p$-FDG} of $(G,\alpha)$. The study of automorphisms of $H$ thus reduces to the one of automorphisms of the single $H_p$, of which we fix one now. By the structure theorem for finite abelian groups, we can write $H_p=\prod_{i=1}^n{\mathbb{Z}/p^{e_i}\mathbb{Z}}$, where the tuple $(e_1,\ldots,e_n)$ is nondecreasing. Let $A=(a_{i,j})_{1\leq i,j\leq n}\in\mathrm{Mat}_n(\mathbb{Z})$. Denoting by $S=\{v_1,\ldots,v_n\}$ the standard generating set of $H_p$, $A$ can be understood as a representation of the function $S\rightarrow H_p$ mapping $v_j\mapsto\sum_{i=1}^n{a_{i,j}v_i}$. The following theorem is mostly a summary of some of Hillar and Rhea's results:

\begin{theoremm}\label{hillarRheaTheo}
With notation as fixed above, the following are equivalent:

(1) The assignment represented by $A$ extends to an endomorphism of $H_p$ (which we also call \textbf{represented by $A$}).

(2) For all $1\leq j\leq i\leq n$, we have $p^{e_i-e_j}\mid a_{i,j}$.

Furthermore, that endomorphism then is an automorphism if and only if $p\nmid\mathrm{det}(A)$.
\end{theoremm}

\begin{proof}
The implication \enquote{(2)$\Rightarrow$(1)} follows immediately from \cite[Theorem 3.3]{HR07a}. As for the implication \enquote{(1)$\Rightarrow$(2)}: If $A$ represents an endomorphism $\varphi$ of $H_p$, then by \cite[Theorem 3.3]{HR07a}, $\varphi$ can also be represented by a matrix $B=(b_{i,j})_{1\leq i,j\leq n}$ such that (2) with $b_{i,j}$ in place of $a_{i,j}$ holds, and (2) then follows since $a_{i,j}\equiv b_{i,j}\hspace{3pt}(\mathrm{mod}\hspace{3pt}p^{e_i})$. The \enquote{Furthermore} only is a reformulation of \cite[Theorem 3.6]{HR07a}.
\end{proof}

We now introduce the following concept, which will be useful in the structural treatment of LCCs in finite abelian $p$-groups:

\begin{deffinition}\label{compatibleDef}
Let $e=(e_1,\ldots,e_n),f=(f_1,\ldots,f_n)$ be nondecreasing tuples of nonnegative integers such that $e_i\geq f_i$ for $i=1,\ldots,n$.

(1) We say that $e$ and $f$ are \textbf{downward compatible} if and only if for all $1\leq j\leq i\leq n$, we have $e_j-f_j\leq e_i-f_i$.

(2) We say that $e$ and $f$ are \textbf{upward compatible} if and only if for all $1\leq j\leq i\leq n$, we have $e_j-f_j\geq e_i-f_i$.
\end{deffinition}

The following lemma, in combination with Transfer Lemma \ref{transferLem}, will allow us to transfer LCCs between finite abelian $p$-groups:

\begin{lemmma}\label{compLem}(\enquote{Compatibility Lemma})
Let $e=(e_1,\ldots,e_n),f=(f_1,\ldots,f_n)$ be nondecreasing tuples of nonnegative integers such that $e_i\geq f_i$ for $i=1,\ldots,n$. Consider the finite abelian $p$-groups $E=\prod_{i=1}^n{\mathbb{Z}/p^{e_i}\mathbb{Z}}$ and $F=\prod_{i=1}^n{\mathbb{Z}/p^{f_i}\mathbb{Z}}$, as well as the projection $\pi:E\rightarrow F$ obtained as the product of the canonical projections $\pi_i:\mathbb{Z}/p^{e_i}\mathbb{Z}\rightarrow\mathbb{Z}/p^{f_i}\mathbb{Z}$, and set $N:=\mathrm{ker}(\pi)$.

(1) If $e$ and $f$ are downward compatible (in which case we also call the groups $E$ and $F$ \textbf{downward compatible}), then for every matrix $A\in\mathrm{Mat}_n(\mathbb{Z})$ representing an automorphism $\alpha$ of $E$, $A$ also represents an automorphism $\tilde{\alpha}$ of $F$. Furthermore, $N$ is $\alpha$-admissible, and $\tilde{\alpha}$ is the automorphism of $F$ induced by $\alpha$.

(2) If $e$ and $f$ are upward compatible (in which case we also call the groups $E$ and $F$ \textbf{upward compatible}), then for every matrix $A\in\mathrm{Mat}_n(\mathbb{Z})$ representing an automorphism $\tilde{\alpha}$ of $F$, $A$ also represents an automorphism $\alpha$ of $E$. Furthermore, $N$ is $\alpha$-admissible, and $\tilde{\alpha}$ is the automorphism of $F$ induced by $\alpha$.
\end{lemmma}

\begin{proof}
We only show (1), as the proof of (2) is analogous. Let $A=(a_{i,j})_{1\leq i,j\leq n}$. By assumption and Theorem \ref{hillarRheaTheo}, for all $1\leq j\leq i\leq n$, we have $p^{e_i-e_j}\mid a_{i,j}$ and $p\nmid\mathrm{det}(A)$. Now for all $1\leq j\leq i\leq n$, we have $f_i-f_j\leq e_i-e_j$, since this is equivalent to $e_j-f_j\leq e_i-f_i$, which holds by assumption. Hence for all $1\leq j\leq i\leq n$, we find that $p^{f_i-f_j}\mid a_{i,j}$, whence $A$ indeed defines an automorphism $\tilde{\alpha}$ of $F$, and it is immediate to check that $\pi\circ\alpha=\tilde{\alpha}\circ\pi$, which implies the rest of the statement.
\end{proof}

\subsection{Rational canonical forms and the elementary abelian case}\label{subsec2P4}

As announced in Subsection \ref{subsec1P2}, we will now treat the elementary abelian case using primary rational canonical forms, following an idea originally by Elspas \cite{Els59a} for treating the prime field case and later successfully applied by Hern{\'a}ndez-Toledo \cite{Her05a} to describe the dynamics of automorphisms of finite vector spaces in general as well. Recall that for every field $K$, each matrix $A\in\mathrm{Mat}_n(K)$ is similar to a unique block diagonal matrix whose blocks $A_1,\ldots,A_s$ are of the form $A_j=\mathrm{M}(Q_j(X))$, where $Q_1(X),\ldots,Q_s(X)\in K[X]$ are monic polynomials such that $Q_1(X)\mid Q_2(X)\mid\cdots\mid Q_s(X)$. The $Q_j(X)$, the product of which is the characteristic polynomial of $A$, are called the \textit{elementary divisors of $A$}, and the block diagonal matrix the \textit{rational canonical form of $A$} (or \textit{Frobenius normal form of $A$}). The basic theory of rational canonical forms is presented in elaborate form, for example, in \cite[Section 12.2]{DF04a}.

For studying the dynamics of endomorphisms of finite vector spaces, a variant of this is most useful: Factorizing $Q_j(X)=\prod_{l=1}^{t_j}{P_{j,l}(X)^{k_{j,l}}}$, each diagonal block $A_j$ is similar to a block diagonal matrix with blocks $\mathrm{M}(P_{j,l}^{k_{j,l}})$, $l=1,\ldots,t_j$. In total, $A$ is therefore similar to a block diagonal matrix whose blocks $B_1,\ldots,B_r$ are each of the form $B_i=\mathrm{M}(P_i(X)^{k_i})$ for some monic irreducible polynomial $P_i(X)\in K[X]$; $B_i$ is regular if and only if $P_i(0)\not=0$. This representation is unique up to reordering of the blocks, and one speaks of the \textit{primary rational canonical form of $A$}; the blocks of any such representation of $A$, counted with multiplicities, will hitherto be referred to as the \textit{primary Frobenius blocks of $A$}.

From now on, assume that $K$ is finite and $A\in\mathrm{GL}_n(K)$. Then one can give an explicit description of the cycle structure of the action of $A$ on $K^n$ in terms of its primary Frobenius blocks, see \cite[Section 6]{Her05a}. For our purposes, we only need information on the largest cycle length of the action of $A$ on $K^n$, though, which by \cite[Corollary 1]{Hor74a} (or as another consequence of the theory of rational canonical forms) coincides with the order of $A$ in the group $\mathrm{GL}_n(K)$. To this end, recall that for $Q(X)\in K[X]$ with $Q(0)\not=0$, the \textit{order of $Q(X)$}, denoted by $\mathrm{ord}(Q(X))$, is defined as the least positive integer $o$ such that $Q(X)\mid X^o-1$ in $K[X]$ (or equivalently, as the order of $X+(Q(X))$ in the group of units of the $K$-algebra $K[X]/(Q(X))$). We note the following well-known result:

\begin{propposition}\label{orderProp}(see \cite[Theorem 3.8]{LN97a})
Let $K$ be a finite field of characteristic $p$, $P(X)\in K[X]$ monic and irreducible such that $P(0)\not=0$, and let $k\in\mathbb{N}^+$. Let $\alpha\in\overline{K}$ be any root of $P(X)$, and denote by $\mathrm{ord}(\alpha)$ the order of $\alpha$ in $\overline{K}^{\ast}$. Then $\mathrm{ord}(P(X))=\mathrm{ord}(\alpha)\cdot p^{\lceil\mathrm{log}_p(k)\rceil}$.\qed
\end{propposition}

Coming back to our original question on orders of finite vector space automorphisms:

\begin{corrollary}\label{herCor}(follows from \cite[Theorems 4 and 5]{Her05a})
Let $K$ be a finite field.

(1) Let $P(X)\in K[X]$ with $P(0)\not=0$. Then we have $\mathrm{ord}(\mathrm{M}(P(X)))=\mathrm{ord}(P(X))$.

(2) Let $A\in\mathrm{GL}_n(K)$, and say the primary Frobenius blocks of $A$ are the companion matrices of $P_1(X)^{k_1},\ldots,P_r(X)^{k_r}$, with each $P_i(X)\in K[X]$ irreducible. Then we have $\mathrm{ord}(A)=\mathrm{lcm}(\mathrm{ord}(P_1(X)^{k_1}),\ldots,\mathrm{ord}(P_r(X)^{k_r}))$.\qed
\end{corrollary}

We are now ready for a comprehensive discussion of the elementary abelian case:

\begin{theoremm}\label{elAbTheo}
Let $p$ be a prime, $n\in\mathbb{N}^+$ and $\alpha$ an automorphism of $(\mathbb{Z}/p\mathbb{Z})^n$ such that $\lambda(\alpha)\geq\frac{1}{2}$.

\begin{enumerate}
\item If $p$ is odd and $n\not=2$, then with respect to an appropriate basis, $\alpha$ is given by the companion matrix of a monic primitive irreducible polynomial of degree $n$ over $\mathbb{F}_p$. Conversely, all such $\alpha$ satisfy $\lambda(\alpha)=1-\frac{1}{p^n}>\frac{1}{2}$.

\item If $p$ is odd and $n=2$, then with respect to an appropriate basis, $\alpha$ is given either by the companion matrix of a monic primitive irreducible polynomial of degree $2$ over $\mathbb{F}_p$ (and conversely, for such $\alpha$, we have $\lambda(\alpha)=1-\frac{1}{p^2}>\frac{1}{2}$) or by the companion matrix of $(X-g)^2\in\mathbb{F}_p[X]$ for some generator $g$ of $\mathbb{F}_p^{\ast}$ (and conversely, all such $\alpha$ satisfy $\lambda(\alpha)=1-\frac{1}{p}>\frac{1}{2}$).

\item If $p=2$, then one of the following holds:

\begin{itemize}
\item $n=2$ and $\alpha$ is given, with respect to an appropriate basis, by the companion matrix of $(X-1)^2\in\mathbb{F}_2[X]$. Conversely, in this case, $\lambda(\alpha)=\frac{1}{2}$.

\item $n=3$ and $\alpha$ is given, with respect to an appropriate basis, by the companion matrix of $(X-1)^3\in\mathbb{F}_2[X]$. Conversely, in this case, $\lambda(\alpha)=\frac{1}{2}$.

\item $\alpha$ is given, with respect to an appropriate basis, by the companion matrix of a squarefree polynomial $P(X)=\prod_{i=1}^r{P_i(X)}$, where the $P_i(X)$ are pairwise distinct primitive irreducible polynomials over $\mathbb{F}_2$ such that, setting $d_i:=\mathrm{deg}(P_i(X))$ for $i=1,\ldots,r$, the $d_i$ are pairwise coprime and $\prod_{i=1}^r{(1-\frac{1}{2^{d_i}})}\geq\frac{1}{2}$. In this case, $\lambda(\alpha)=\prod_{i=1}^r{(1-\frac{1}{2^{d_i}})}$.
\end{itemize}
\end{enumerate}
\end{theoremm}

\begin{proof}
First note that in any case, since the domain of the largest cycle of $\alpha$ generates the entire group $(\mathbb{Z}/p\mathbb{Z})^n$, $\alpha$ is a cyclic $\mathbb{F}_p$-vector space automorphism, i.e., it can be represented by $\mathrm{M}(P(X))$ for a single monic polynomial $P(X)\in K[X]$ such that $P(0)\not=0$. Factorize $P(X)=\prod_{i=1}^r{P_i(X)^{k_i}}$, set $d_i:=\mathrm{deg}(P_i(X))$, $i=1,\ldots,r$, and note that by Corollary \ref{herCor}(2), $\mathrm{ord}(\alpha)=\mathrm{lcm}(\mathrm{ord}(P_1(X)^{k_1}),\ldots,\mathrm{ord}(P_r(X)^{k_r}))$. Now for $i=1,\ldots,r$, by Proposition \ref{orderProp} and Corollary \ref{herCor}(1), $\mathrm{ord}(P_i(X)^{k_i})$ is a divisor of $(p^{d_i}-1)\cdot p^{\lceil\mathrm{log}_p(k_i)\rceil}$, and in view of Product Lemma \ref{productLem}(2,ii), it cannot be a proper divisor; in particular, $P_i(X)$ must be primitive.

If $p$ is odd, this implies that $2\mid\mathrm{ord}(P_i(X)^{k_i})$, whence by Product Lemma \ref{productLem}(2,ii), we have $r=1$. Furthermore, for odd $p$, we find that $\lceil\mathrm{log}_p(k_1)\rceil\leq k_1-1$ for all $k_1\geq 1$ and $\lceil\mathrm{log}_p(k_1)\rceil\leq k_1-2$ for all $k_1\geq 3$. Using the second bound, we obtain that for $k_1\geq 3$, we have \[\mathrm{ord}(\alpha)<p^{d_1}\cdot p^{k_1-2}=p^{d_1+k_1-2}\leq p^{d_1k_1-1}=\frac{1}{p}\cdot|(\mathbb{Z}/p\mathbb{Z})^{d_1k_1}|,\] a contradiction. For $k_1\leq 2$, using the first bound, we obtain similarly that any case in which $d_1+k_1\leq d_1k_1$ holds is contradictory. However, this inequality is satisfied whenever both $k_1$ and $d_1$ are greater than $1$, so if $k_1=2$, then necessarily $d_1=1$, whence $\alpha$ then is represented by the companion matrix of $(X-g)^2$ for some generator $g$ of $\mathbb{F}_p^{\ast}$ (as all primitive polynomials of degree $1$ over $\mathbb{F}_p$ are of that form), and this indeed yields an automorphism of $(\mathbb{Z}/p\mathbb{Z})^2$ of $\lambda$-value $1-\frac{1}{p}$. The case $k_1=1$ corresponds to all the \enquote{standard} examples with $\lambda$-value $1-\frac{1}{p^{d_1}}$.

We now conclude by discussing the case $p=2$. First, we will treat the special case where $P(X)$ is a power $P_1(X)^{k_1}$ of an irreducible polynomial $P_1(X)$, with $d_1:=\mathrm{deg}(P_1(X))$. Since for all $k_1\geq 1$, we have the inequality $\lceil\mathrm{log}_2(k_1)\rceil\leq k_1-1$, we can conclude as before that every case where $d_1+k_1\leq d_1k_1$ is contradictory; in particular, $d_1$ and $k_1$ cannot both be greater than $1$. If $k_1\geq 2$ (and hence $d_1=1$), we find that $\mathrm{ord}(\alpha)=2^{\lceil\mathrm{log}_2(k_1)\rceil}$, whence for $k_1\geq 4$, we have $\mathrm{ord}(\alpha)\leq 2^{k_1-2}=\frac{1}{4}\cdot|(\mathbb{Z}/2\mathbb{Z})^{k_1}|$, a contradiction. On the other hand, for $k_1=2$ or $k_1=3$, we do obtain automorphisms of $(\mathbb{Z}/2\mathbb{Z})^2$ and $(\mathbb{Z}/2\mathbb{Z})^3$ respectively with $\lambda$-value precisely $\frac{1}{2}$, as described in points (i) and (ii) above. This concludes the analysis of the special case.

However, that analysis implies via Product Lemma \ref{productLem}(2,ii) that if $P(X)$ has more than one factor $P_i(X)^{k_i}$, then no $k_i$ can be greater than $1$, so $P(X)$ is square-free, and its factors are monic primitive irreducible polynomials. In view of Product Lemma \ref{productLem}(2,ii), their orders $2^{d_i}-1$ must be pairwise coprime, which is equivalent to the pairwise coprimality of the degrees $d_i$, and conversely, in this case, the $\lambda$-value of $\alpha$ equals $\prod_{i=1}^r{(1-\frac{1}{2^{d_i}})}$.
\end{proof}

We also note the following consequence of Theorem \ref{elAbTheo}:

\begin{corrollary}\label{properCor}
Let $(G,\alpha)$ be a periodic elementary abelian $2$-FDG such that $|G|\geq 2^4$ and $\lambda(\alpha)\geq\frac{1}{2}$. Then $\lambda(\alpha)>\frac{1}{2}$.
\end{corrollary}

\begin{proof}
By Theorem \ref{elAbTheo}, there exist $r\in\mathbb{N}^+$ and pairwise coprime integers $d_1,\ldots,d_r\geq 2$ such that $\frac{1}{2}=\lambda(\alpha)=\prod_{i=1}^r{\frac{2^{d_i}-1}{2^{d_i}}}$. Since the numerators $2^{d_i}-1$ are odd and the denominators $2^{d_i}$ have no odd prime divisors, no cancelations are necessary in order to pass to the canceled form $\frac{1}{2}$, whence we conclude that $1=\prod_{i=1}^r{(2^{d_i}-1)}$, a contradiction.
\end{proof}

\begin{remmark}\label{strengthenRem}
In Subsection \ref{subsec2P7}, we will see (Lemma \ref{anLem}(1)) that under the assumptions of Corollary \ref{properCor}, we can even deduce $\lambda(\alpha)>\frac{9}{16}$.
\end{remmark}

\subsection{Classification of periodic abelian FDGs \texorpdfstring{$(G,\alpha)$}{(G,alpha)} with \texorpdfstring{$\lambda(\alpha)>\frac{1}{2}$}{lambda(alpha)>1/2}}\label{subsec2P5}

Recall that for a finite $p$-group $G$ and $i\in\mathbb{N}$, the \textit{$i$-th omega subgroup of $G$} is defined as $\Omega_i(G):=\langle g\in G\mid g^{p^i}=1_G\rangle$, and the \textit{$i$-th agemo subgroup of $G$} is defined as $\mho^i(G):=\langle g^{p^i}\mid g\in G\rangle$. We begin by observing that each nondecreasing tuple $(e_1,\ldots,e_n)$ is upward compatible with the constant tuple $(e_n,\ldots,e_n)$. Hence we will first try to better understand automorphisms with $\lambda$-value at least $\frac{1}{2}$ of groups of the form $(\mathbb{Z}/p^{e_n}\mathbb{Z})^n$. The case $e_n=1$ was treated in Subsection \ref{subsec2P4}, and luckily, the case $e_n\geq 2$ is much easier, as there are no such automorphisms in this case for $n\geq 2$:

\begin{propposition}\label{boundProp}
Let $p$ be a prime, $n,e_n\in\mathbb{N}^+$. Then $\mathrm{mao}((\mathbb{Z}/p^{e_n}\mathbb{Z})^n)\leq p^{e_n-1}(p^n-1)$. In particular, $\lambda((\mathbb{Z}/p^{e_n}\mathbb{Z})^n)<\frac{1}{2}$ whenever $n,e_n\geq 2$.
\end{propposition}

\begin{proof}
We proceed by induction on $e_n$, the induction base $e_n=1$ following immediately from \cite[Theorem 2]{Hor74a}. Now assume $e_n\geq 2$ and set $G:=(\mathbb{Z}/p^{e_n}\mathbb{Z})^n$. Note that $\mho^1(G)\cong(\mathbb{Z}/p^{e_n-1}\mathbb{Z})^n$, so by the induction hypothesis, the restriction $\overline{\alpha}$ of $\alpha$ to $\mho^1(G)$ has order $o$ bounded by $p^{e_n-2}(p^n-1)$. It is not difficult to see that $\alpha^o=\mathrm{id}+f$ for some homomorphism $f:G\rightarrow\Omega_1(G)\cong(\mathbb{Z}/p\mathbb{Z})^n$. Since every element of order $p$ in $G$ is a $p$-fold, we have $f^2=f\circ f=0$, whence $\alpha^{p\cdot o}=(\mathrm{id}+f)^p=\mathrm{id}+p\cdot f=\mathrm{id}$, and so $\mathrm{ord}(\alpha)\leq p\cdot o\leq p^{e_n-1}(p^n-1)$. For the \enquote{In particular}, note that if $n,e_n\geq 2$, then \[\frac{p^{e_n-1}(p^n-1)}{p^{n\cdot e_n}}<\frac{p^{n+e_n-1}}{p^{n\cdot e_n}}=p^{n+e_n-n\cdot e_n-1}\leq p^{-1}\leq\frac{1}{2}.\]
\end{proof}

We can now finally tackle the $\frac{1}{2}$-LCC in finite abelian groups:

\begin{theoremm}\label{lccAbTheo}
Let $G$ be a finite abelian $p$-group for some prime $p$ such that $\lambda(G)\geq\frac{1}{2}$. Then either $G$ is elementary abelian or primary cyclic or $G\cong\mathbb{Z}/2\mathbb{Z}\times\mathbb{Z}/4\mathbb{Z}$.
\end{theoremm}

\begin{proof}
Write $G=\prod_{i=1}^n{\mathbb{Z}/p^{e_i}\mathbb{Z}}$ with $(e_1,\ldots,e_n)$ nondecreasing and fix an automorphism $\tilde{\alpha}$ of $G$ such that $\lambda(\tilde{\alpha})\geq\frac{1}{2}$. We will show under the assumption $n,e_n\geq 2$ that $p=2$ and $G$ is isomorphic to $\mathbb{Z}/2\mathbb{Z}\times\mathbb{Z}/4\mathbb{Z}$.

First, assume that $e_{n-1}>1$. Then since $(e_1,\ldots,e_n)$ is upward compatible with $(e_n,\ldots,e_n)$, by Compatibility Lemma \ref{compLem}(2), $\tilde{\alpha}$ is induced by some automorphism $\alpha$ of $(\mathbb{Z}/p^{e_n}\mathbb{Z})^n$. By Proposition \ref{boundProp}, the order of $\alpha$, and hence also the order of $\tilde{\alpha}$, is bounded by $p^{e_n-1}(p^n-1)$. However, \[p^{e_n-1}(p^n-1)<p^{e_n-1}\cdot p^n=\frac{1}{p}p^{e_n}\cdot p^n\leq \frac{1}{p}p^{e_n}p^{e_1+\cdots+e_{n-1}}=\frac{1}{p}|G|,\] a contradiction. Hence $G$ is of the form $(\mathbb{Z}/p\mathbb{Z})^{n-1}\times\mathbb{Z}/p^{e_n}\mathbb{Z}$ for some $e_n\geq 2$.

Our next goal is to exclude the case $n\geq 3$. First, assume that $p$ is odd. Just as before, we can consider an automorphism $\alpha$ of $(\mathbb{Z}/p^{e_n}\mathbb{Z})^n$ inducing $\tilde{\alpha}$ on the quotient $G$. Now the order of $\alpha$ is bounded below by $\frac{1}{2}p^n\cdot p^{e_n-1}$, and following the proof of Proposition \ref{boundProp}, we find that defining $\alpha_i$ to be the restriction of $\alpha$ to $\mho^i((\mathbb{Z}/p^{e_n}\mathbb{Z})^n)$ for $i=0,\ldots,e_n-1$, for each $i=0,\ldots,e_n-2$, either $\mathrm{ord}(\alpha_i)=p\cdot\mathrm{ord}(\alpha_{i+1})$ or $\mathrm{ord}(\alpha_i)=\mathrm{ord}(\alpha_{i+1})$. However, if the second equality holds just once, then $\mathrm{ord}(\alpha_{e_n-1})\geq\frac{1}{2}p^n\cdot p>p^n$, a contradiction. Hence the order decreases by the factor $p$ in each step and $\mathrm{ord}(\alpha_{e_n-1})=\frac{1}{p^{e_n-1}}\mathrm{ord}(\alpha)\geq\frac{1}{2}p^n$, whence $\alpha_{e_n-1}$ is an automorphism of $\mho^{e_n-1}((\mathbb{Z}/p^{e_n}\mathbb{Z})^n)\cong(\mathbb{Z}/p\mathbb{Z})^n$ of $\lambda$-value at least $\frac{1}{2}$. By Theorem \ref{elAbTheo} and the assumptions $n\geq 3$ and $p>2$, this implies that $\mathrm{ord}(\alpha_{e_n-1})=p^n-1$, and hence $\mathrm{ord}(\alpha)=p^{e_n-1}(p^n-1)$. Therefore, by \cite[Corollary 1]{Hor74a}, $\alpha$ has a largest cycle of length $p^{e_n-1}(p^n-1)$, and that cycle induces a largest cycle of $\tilde{\alpha}$ of length a divisor of that number. But the largest cycle of $\tilde{\alpha}$ has length bounded below by $\frac{1}{2}p^{n-1}p^{e_n}>\frac{1}{2}p^{e_n-1}(p^n-1)$, whence its length must equal $p^{e_n-1}(p^n-1)$. However, elements on an automorphism cycle are all of the same order, and it is not difficult to see that the largest subset of $G$ consisting of elements of a fixed order is the subset of elements of order $p^{e_n}$ and only has size $p^{n-1}\cdot\phi(p^{e_n})=p^{n-1}p^{e_n-1}(p-1)=p^{e_n-1}(p^n-p^{n-1})<p^{e_n-1}(p^n-1)$, a contradiction. This concludes the argument that $n\geq 3$ is not possible for odd $p$.

We will now refute the assumption $n\geq 3$ for $p=2$ as well. First, note that in the case $p=2$, we are generally in a more convenient position, as we can immediately refute the case $e_n\geq 4$, since in that case, $(1,\ldots,1,e_n)$ is downward compatible with $(0,\ldots,0,3)$, so that by Compatibility Lemma \ref{compLem}(1) and Transfer Lemma \ref{transferLem}(3,a), the assumption $\lambda(G)\geq\frac{1}{2}$ implies $\lambda(\mathbb{Z}/8\mathbb{Z})\geq\frac{1}{2}$, a contradiction. We can also refute the case $n\geq 4$, as follows: The largest number of elements of the same order in $G$ is the one of elements of order $2^{e_n}$, of which there are precisely $2^{n-1}\cdot 2^{e_n-1}=\frac{1}{2}|G|$. Hence $\Lambda(\tilde{\alpha})=2^{n-1}2^{e_n-1}$. Let $\alpha$ be an automorphism of $(\mathbb{Z}/2^{e_n}\mathbb{Z})^n$ inducing $\tilde{\alpha}$. Then the order of $\alpha$ is bounded below by $2^{n-1}\cdot 2^{e_n-1}$, and applying the same argument of descending through the successive agemo subgroups as before, we find that the order of the restriction $\overline{\alpha}$ of $\alpha$ to the the $(e_n-1)$-th agemo subgroup is bounded below by $2^{n-1}$, whence by Corollary \ref{properCor}, its actual order is of the form $\kappa\cdot 2^{n-1}$ for some $\kappa\in\left(1,2\right)$, and the order of $\alpha$ equals $\kappa\cdot 2^{n-1}\cdot 2^{e_n-1}$, contradicting the fact that it is an integer multiple of the order of $\tilde{\alpha}$, which by \cite[Corollary 1]{Hor74a} equals $\Lambda(\tilde{\alpha})$. Hence all that is left to refute the case $n\geq 3$ for $p=2$ as a whole is to refute the two groups $(\mathbb{Z}/2\mathbb{Z})^2\times\mathbb{Z}/4\mathbb{Z}$ and $(\mathbb{Z}/2\mathbb{Z})^2\times\mathbb{Z}/8\mathbb{Z}$, which can easily be done with GAP \cite{GAP4}: The first group has $\lambda$-value $\frac{3}{8}$, and the second $\frac{3}{16}$.

We are hence left with the case $n=2$, i.e., $G\cong\mathbb{Z}/p\mathbb{Z}\times\mathbb{Z}/p^{e_2}\mathbb{Z}$ with $e_2\geq 2$. It is then not difficult to show that $p=2$ implies $G\cong\mathbb{Z}/2\mathbb{Z}\times\mathbb{Z}/4\mathbb{Z}$, since, as noted before, $e_2\leq 3$, and the group $\mathbb{Z}/2\mathbb{Z}\times\mathbb{Z}/8\mathbb{Z}$ is refuted with GAP \cite{GAP4} again: $\lambda(\mathbb{Z}/2\mathbb{Z}\times\mathbb{Z}/8\mathbb{Z})=\frac{1}{4}$.

So all that is left is to exclude the case $n=2$ for odd $p$. We will do this \enquote{by hand}. Let $\pi:\mathbb{Z}\rightarrow\mathbb{Z}/p\mathbb{Z}$ and $\pi':\mathbb{Z}\rightarrow\mathbb{Z}/p^{e_2}\mathbb{Z}$ denote the canonical projections. By Theorem \ref{hillarRheaTheo}, $\alpha$ is represented by a matrix $A\in\mathrm{Mat}_2(\mathbb{Z})$ such that the coefficient in the bottom left corner is divisible by $p^{e_2-1}$, say \[A=\begin{pmatrix}a & b \\ p^{e_2-1}c & d\end{pmatrix}\] for $a,b,c,d\in\mathbb{Z}$. Since necessarily $p\nmid\mathrm{det}(A)$, we can conclude that $p\nmid a$ and $p\nmid d$. We will show that all cycle lengths of $\alpha$ divide the number $n_0:=p^{e_2-1}(p-1)$ and hence are too small for $\lambda(\alpha)$ to be in $\left[\frac{1}{2},1\right]$. Since for any group $H$, any automorphism $\beta$ of $H$ and all elements $x,y\in H$, the length of the cycle of $xy$ under $\beta$ divides the least common multiple of the cycle lengths of $x$ and $y$ under $\beta$, it suffices to show that the cycle length of each of the generators $(\pi(1),\pi'(0))$ and $(\pi(0),\pi'(1))$ under $\alpha$ divides $n_0$. For this, we use the following formulas, which are easy to show by induction on $n\in\mathbb{N}$: \begin{equation}\label{form1}\alpha^n(\pi(1),\pi'(0))=(\pi(a)^n,p^{e_2-1}\pi'(c\sum_{k=0}^{n-1}{a^{n-1-k}d^k})),\end{equation}\begin{equation}\label{form2}\alpha^n(\pi(0),\pi'(1))=(\pi(b\sum_{k=0}^{n-1}{a^{n-1-k}d^k}),\pi'(d)^n+p^{e_2-1}\pi'(bc\sum_{l=1}^{n-1}{la^{n-1-l}d^{l-1}})).\end{equation}

We distinguish two cases. First, assume that $a\equiv d\hspace{3pt}(\mathrm{mod}\hspace{3pt}p)$. Then by (\ref{form1}), we have \[\alpha^{n_0}(\pi(1),\pi'(0))=(\pi(a)^{n_0},p^{e_2-1}\pi'(c\cdot n_0a^{n_0-1}))=(\pi(1),\pi'(0)),\] and by (\ref{form2}), we have \[\alpha^{n_0}(\pi(0),\pi'(1))=(\pi(bn_0a^{n_0-1}),\pi'(d)^{n_0}+p^{e_2-1}\pi'(bc\cdot a^{n_0-2}\cdot\frac{1}{2}(n_0-1)n_0))=(\pi(0),\pi'(1)).\]

So from now on, we may assume that $a\not\equiv d\hspace{3pt}(\mathrm{mod}\hspace{3pt}p)$. Then we can write \[\pi(\sum_{k=0}^{n_0-1}{a^{n_0-1-k}d^k})= \pi(a)^{n_0-1}\sum_{k=0}^{n_0-1}{(\frac{\pi(d)}{\pi(a)})^k}=\pi(a)^{n_0-1}\cdot\frac{(\frac{\pi(d)}{\pi(a)})^{n_0}-1}{\frac{\pi(d)}{\pi(a)}-1}=\pi(0),\] from which it follows immediately that $\alpha^{n_0}(\pi(1),\pi'(0))=(\pi(1),\pi'(0))$ and that the first entry of $\alpha^{n_0}(\pi(0),\pi'(1))$ equals $\pi(0)$. As for its second entry, we note that it can be written as \[\pi'(1)+p^{e_2-1}\pi'(bca^{n_0-2})\sum_{l=1}^{n_0-1}{l(\frac{\pi'(d)}{\pi'(a)})^{l-1}}.\] In order to show that this is equal to $\pi'(1)$, we argue that the sum \[\sum_{l=1}^{n_0-1}{l(\frac{\pi'(d)}{\pi'(a)})^{l-1}}\in\mathbb{Z}/p^{e_2}\mathbb{Z}\] is not a unit in the ring $\mathbb{Z}/p^{e_2}\mathbb{Z}$. However, since the canonical projection $\pi'':\mathbb{Z}/p^{e_2}\mathbb{Z}\rightarrow\mathbb{Z}/p\mathbb{Z}$ satisfies $\pi=\pi''\circ\pi'$ and sends units to units, this follows immediately from the following, working with the formal derivative in the function field $(\mathbb{Z}/p\mathbb{Z})(X)$: \[\sum_{l=1}^{n_0-1}{l(\frac{\pi(d)}{\pi(a)})^{l-1}}=(\frac{\mathrm{d}}{\mathrm{d}X}{\frac{X^{n_0}-1}{X-1}})(\frac{\pi(d)}{\pi(a)})=(\frac{-X^{n_0}+1}{(X-1)^2})(\frac{\pi(d)}{\pi(a)})=\pi(0).\]
\end{proof}

Note that now, we could deduce Corollary \ref{strictCor} from Theorem \ref{abelianTheo}:

\begin{corrollary}\label{abStrictCor}
The list given in Corollary \ref{strictCor} is the complete list of \textbf{abelian} periodic FDGs with $\lambda$-value greater than $\frac{1}{2}$, together with their $\lambda$-values and the information whether they are $\lambda$-maximal or not.
\end{corrollary}

\begin{proof}
If $(G,\alpha)$ is an abelian periodic FDG with $\lambda$-value greater than $\frac{1}{2}$, then by Product Lemma \ref{productLem}(2,iii), the Sylow FDGs $(G_p,\alpha_p)$ of $(G,\alpha)$ also have $\lambda$-value greater than $\frac{1}{2}$, whence by Theorem \ref{lccAbTheo}, they each are either elementary abelian or primary cyclic. In the latter case, which can only occur for odd $p$, by cyclicity of the automorphism group of primary cylic groups of odd order, it is clear that $\alpha_p$ must be the multiplication by a primitive root modulo $p^k$, whence in particular $2\mid\Lambda(\alpha_p)$, and the first case, for odd $p$, $(G_p,\alpha_p)$ is, up to isomorphism, as in points 1 and 2 of the list from Theorem \ref{elAbTheo}, whence also $2\mid\Lambda(\alpha_p)$. It follows by Product Lemma \ref{productLem}(2,ii) that $|G|$ has at most one odd prime divisor. So precisely one of the following cases occurs:

(i) $G$ is a (possibly trivial) elementary abelian $2$-group. Then by Theorem \ref{elAbTheo}, $(G,\alpha)$ is isomorphic to an FDG as in point 1 of the list. The assertion on $\lambda$-maximality is easy to check.

(ii) $G$ is a nontrivial primary cyclic $p$-group for an odd prime $p$. Then $(G,\alpha)$ is, up to isomorphism, as in point 2. Again, the assertion on $\lambda$-maximality is clear.

(iii) $G$ is a nontrivial elementary abelian $p$-group for an odd prime $p$. Then by Theorem \ref{elAbTheo}, $(G,\alpha)$ is isomorphic to an FDG from points 3 and 4 of the list, and in both cases, the assertions on $\lambda$-maximality are clear.

(iv) $|G|=2^kp^m$ for some odd prime $p$ and $k,m\in\mathbb{N}^+$. Then $G$ is, up to isomorphism, as in point 5 of the list, with the \enquote{compatibility restrictions} on the two factors as well as the assertion on $\lambda$-maximality following from Product Lemma \ref{productLem}(2,ii and iii).
\end{proof}

\subsection{Classification of periodic abelian gFDGs \texorpdfstring{$(G,A)$}{(G,A)} with \texorpdfstring{$\lambda(A)=1$}{lambda(A)=1}}\label{subsec2P6}

For proving Theorem \ref{equalTheo}, we will need this classification result at one point. Just as in the last subsection, in view of Product Lemma \ref{productLem}(2,iv), the main effort lies in analyzing the $p$-group level. We first give a classification of the finite abelian $p$-groups $G$ with $\lambda_{\mathrm{aff}}(G)=1$:

\begin{theoremm}\label{affineTheo}
Let $G$ be a finite abelian $p$-group for some prime $p$. The following are equivalent:

(1) $\lambda_{\mathrm{aff}}(G)=1$.

(2) $G\cong(\mathbb{Z}/2\mathbb{Z})^2$ or $G\cong\mathbb{Z}/p^n\mathbb{Z}$ for some $n\in\mathbb{N}$.
\end{theoremm}

\begin{proof}[Proof of Theorem \ref{affineTheo}]
For \enquote{(2)$\Rightarrow$(1)}: It is easy to check that if $\alpha$ is an automorphism of order $2$ of $(\mathbb{Z}/2\mathbb{Z})^2$ and $v\in(\mathbb{Z}/2\mathbb{Z})^2$ is not a fixed point of $\alpha$, then $\lambda(\mathrm{A}_{v,\alpha})=1$. For the cyclic groups $G$, just consider the map $G\rightarrow G$, $g\mapsto g_0g$, for any generator $g_0$ of $G$.

For \enquote{(1)$\Rightarrow$(2)}: The idea is to use a downward compatibility argument to reduce to the study of a few special cases. More precisely, write $G=\prod_{i=1}^n{\mathbb{Z}/p^{e_i}\mathbb{Z}}$ with $(e_1,\ldots,e_n)$ nondecreasing. $G$ is downward compatible with the elementary abelian group $(\mathbb{Z}/p\mathbb{Z})^n$, and if $n=2$ and $e_2>1$, then $G$ is either downward compatible with $\mathbb{Z}/p\mathbb{Z}\times\mathbb{Z}/p^2\mathbb{Z}$ (if $e_1<e_2$) or with $(\mathbb{Z}/p^2\mathbb{Z})^2$ (if $e_1=e_2$). Hence by Compatibility Lemma \ref{compLem}(1) and Transfer Lemma \ref{transferLem}(3,b), it is sufficient to show the following three assertions:

(A) For odd $p$ and $n\geq 2$, $\lambda_{\mathrm{aff}}((\mathbb{Z}/p\mathbb{Z})^n)<1$.

(B) For $n\geq 3$, $\lambda_{\mathrm{aff}}((\mathbb{Z}/2\mathbb{Z})^n)<1$.

(C) $\mathrm{max}(\lambda_{\mathrm{aff}}(\mathbb{Z}/2\mathbb{Z}\times\mathbb{Z}/4\mathbb{Z}),\lambda_{\mathrm{aff}}((\mathbb{Z}/4\mathbb{Z})^2))<1$.

Assertion (C) (more precisely, $\lambda_{\mathrm{aff}}(\mathbb{Z}/2\mathbb{Z}\times\mathbb{Z}/4\mathbb{Z})=\lambda_{\mathrm{aff}}((\mathbb{Z}/4\mathbb{Z})^2)=\frac{1}{2}$) is readily verified using GAP \cite{GAP4}. As for assertions (A) and (B), let $p$ be any prime and assume that $n\geq 2$. Note that by Lemma \ref{affineLem}(1), if $A=\mathrm{A}_{v,\alpha}$ is a periodic affine map of $(\mathbb{Z}/p\mathbb{Z})^n$ of $\lambda$-value $1$, then $\alpha$ must have a nontrivial fixed point and $\mathrm{ord}(\alpha)=p^{n-1}$. In particular, the order of one of the primary Frobenius blocks of $\alpha$ equals $p^{n-1}$. Fix one such block, which by Proposition \ref{orderProp} must be the companion matrix of $(X-1)^k$ for some $k\in\mathbb{N}^+$ such that $\lceil\mathrm{log}_p(k)\rceil=n$. Since $\lceil\mathrm{log}_p(n-1)\rceil<n-1$ for all $n\geq 2$, we must have $k=n$ (in particular, $\alpha$ has precisely one primary Frobenius block), and because $\lceil\mathrm{log}_p(n)\rceil<n-1$ holds for all $p$ for $n\geq 4$ and for odd $p$ for $n\geq 3$, we can refute the case $n\geq 3$ for all $p$ except possibly for the subcase $p=2,n=3$. However, that subcase is readily refuted with GAP \cite{GAP4}: $\lambda_{\mathrm{aff}}((\mathbb{Z}/2\mathbb{Z})^3)=\frac{7}{8}$. This proves assertion (B), and for proving assertion (A), it remains to refute the case $n=2$ for odd $p$, which can be done as follows: We already know that, with respect to an appropriate $\mathbb{F}_p$-basis of $(\mathbb{Z}/p\mathbb{Z})^2$, $\alpha$ is represented by $\mathrm{M}((X-1)^2)$, which has characteristic polynomial $(X-1)^2$, so $\alpha$ can be represented, with respect to some other basis $\mathcal{B}$, by a matrix of the form \[\begin{pmatrix}1 & a \\ 0 & 1\end{pmatrix}\] for an element $a\in\mathbb{F}_p^{\ast}$. Replacing $A$ by an appropriate power, we may even assume w.l.o.g.~that $a=1$. From now on, all elements of $(\mathbb{Z}/p\mathbb{Z})^2$ are represented by their coordinate vectors with respect to $\mathcal{B}$. If $v=(x,y)^t$, we find by an easy induction on $n\geq 0$ that the image of $(0,0)^t$ under $A^n$ equals $(nx+\Delta_{n-1}y,ny)^t$, where $\Delta_k=\frac{1}{2}k(k+1)$ is the $k$-th triangle number. In particular, the cycle returns to the zero element after $p$ iterations of $A$, although it should only return after $p^2$ iterations, a contradiction.
\end{proof}

It is now not difficult to even classify the periodic abelian gFDGs $(G,A)$ such that $\lambda(A)=1$:

\begin{corrollary}\label{affineCor}
Let $G$ be a finite abelian group, $A=\mathrm{A}_{g_0,\alpha}$ a periodic affine map of $G$. The following are equivalent:

(1) $\lambda(A)=1$.

(2) Up to isomorphism of gFDGs, one of the following holds:

(a) $G=\prod_{p\mid |G|}{\mathbb{Z}/p^{k_p}\mathbb{Z}}$ is cyclic, $g_0$ is a generator of $G$, and, splitting $\alpha=\prod_{p\mid |G|}{\alpha_p}$ over the Sylow subgroups $\mathbb{Z}/p^{k_p}\mathbb{Z}$ of $G$, $\alpha_p$ is the multiplication by a factor $a_p$ such that $a_p\equiv 1\hspace{3pt}(\mathrm{mod}\hspace{3pt}p)$ for all prime divisors $p$ of $|G|$, and such that, if $k_2\geq 2$, then $a_2\equiv 1\hspace{3pt}(\mathrm{mod}\hspace{3pt}4)$.

(b) $G=(\mathbb{Z}/2\mathbb{Z})^2\times\mathbb{Z}/o\mathbb{Z}$ for some odd $o\geq 1$, and, splitting $g_0=(g_2,g_o)$ and $\alpha=\alpha_2\times\alpha_o$ over that product, $\mathrm{A}_{g_o,\alpha_o}$ is an affine map of $\mathbb{Z}/o\mathbb{Z}$ as described in (a), and $\alpha_2$ is an automorphism of $(\mathbb{Z}/2\mathbb{Z})^2$ of order $2$ such that $g_2\notin\mathrm{fix}(\alpha_2)$.
\end{corrollary}

\begin{proof}
By Product Lemma \ref{productLem}(2,iv) and Theorem \ref{affineTheo}, it suffices to classify the periodic gFDGs $(G,A)$ where $\lambda(A)=1$ and $G$ is of one of the two forms $\mathbb{Z}/p^k\mathbb{Z}$, for $p$ prime and $k\geq 0$, and $(\mathbb{Z}/2\mathbb{Z})^2$.

For $(\mathbb{Z}/2\mathbb{Z})^2$, this is an easy exercise, and for the $\mathbb{Z}/p^k\mathbb{Z}$, a classification as desired is given by Knuth's theorem \cite[Section 3.2.1]{Knu98a}.
\end{proof}

\subsection{Some lower bounds on \texorpdfstring{$\lambda$}{lambda}-values of automorphisms of \texorpdfstring{$(\mathbb{Z}/2\mathbb{Z})^n$}{(Z/2Z)^n}}\label{subsec2P7}

The results from this subsection will mainly be used in the proof of Theorem \ref{denseTheo}, although Lemma \ref{anLem}(1) is also invoked at one point in the proof of Theorem \ref{abelianTheo}.

\begin{lemmma}\label{expLem}
Let $m\in\mathbb{N}^+$. Then \[\prod_{n\geq m}{(1-\frac{1}{2^n})}\geq\mathrm{exp}(-\frac{1}{2^{m-2}}).\]
\end{lemmma}

\begin{proof}
\[\prod_{n\geq m}{(1-\frac{1}{2^n})^{-1}}=\prod_{n\geq m}{(1+\frac{1}{2^n-1})}\leq\prod_{n\geq m}{\mathrm{exp}(\frac{1}{2^n-1})}\leq\prod_{n\geq m}{\mathrm{exp}(\frac{1}{2^{n-1}})}=\]\[\mathrm{exp}(\sum_{n\geq m}{\frac{1}{2^{n-1}}})=\mathrm{exp}(\frac{1}{2^{m-2}}).\]
\end{proof}

All bounds needed henceforth are summarized in the following lemma:

\begin{lemmma}\label{anLem}
Let $G$ be a finite elementary abelian $2$-group. Assume that $\alpha$ is an automorphism of $G$ with $\lambda(\alpha)>\frac{1}{2}$, so that by Theorem \ref{elAbTheo}, there exist $r\in\mathbb{N}^+$ and pairwise distinct monic primitive irreducible polynomials $P_1(X),\ldots,P_r(X)\in\mathbb{F}_2[X]$ such that the primary Frobenius blocks of $\alpha$ are the $\mathrm{M}(P_i(X))$ for $i=1,\ldots,r$. Set $d_i:=\mathrm{deg}(P_i(X))$ for $i=1,\ldots,r$. Then the following hold:

(1) $\lambda(\alpha)>\prod_p{(1-\frac{1}{2^p})}=\frac{5}{4}\cdot\rho_0>0.63038>\frac{9}{16}$, with $\rho_0$ as in Theorem \ref{denseTheo}.

From now on, assume $d_i>2$ for $i=1,\ldots,r$.

(2) $\lambda(\alpha)>(1-\frac{1}{2^4})\cdot\prod_{p\geq 3}{(1-\frac{1}{2^p})}>0.78797$.

(3) If $3,4\notin\{d_1,\ldots,d_r\}$, then $\lambda(\alpha)>(1-\frac{1}{2^6})\cdot\prod_{p\geq5}{(1-\frac{1}{2^p})}$.

(4) If $3\notin\{d_1,\ldots,d_r\}$, then $\lambda(\alpha)>(1-\frac{1}{2^4})(1-\frac{1}{2^9})\cdot\prod_{p\geq 5}{(1-\frac{1}{2^p})}$.

(5) If $4,5\notin\{d_1,\ldots,d_r\}$, then $\lambda(\alpha)>\prod_{n=3,8,25}{(1-\frac{1}{2^n})}\cdot\prod_{p\geq 7}{(1-\frac{1}{2^p})}$.

(6) If $4,7\notin\{d_1,\ldots,d_r\}$, then $\lambda(\alpha)>\prod_{n=3,5,8,49}{(1-\frac{1}{2^n})}\cdot\prod_{p\geq 11}{(1-\frac{1}{2^p})}=\frac{9}{8}\cdot\rho_1$, with $\rho_1$ as in Theorem \ref{denseTheo}.

(7) If $4\notin\{d_1,\ldots,d_r\}$, then $\lambda(\alpha)>(1-\frac{1}{2^8})\cdot\prod_{p\geq 3}{(1-\frac{1}{2^p})}$.

(8) If $5\notin\{d_1,\ldots,d_r\}$, then $\lambda(\alpha)>\prod_{n=3,4,25}{(1-\frac{1}{2^n})}\prod_{p\geq 7}{(1-\frac{1}{2^p})}$.
\end{lemmma}

\begin{proof}
For (1): In order to see that $\lambda(\alpha)>\prod_p{(1-\frac{1}{2^p})}$, proceed as follows: We know that $\lambda(\alpha)=\prod_{i=1}^r{(1-\frac{1}{2^{d_i}})}$, and the $d_i$ are greater than $1$ and pairwise coprime. This allows us to injectively assign to each $d_i$ a value $p_i\leq d_i$ among the \enquote{admissible} exponents $p$ occurring in the denominators of the factors $(1-\frac{1}{2^p})$ of the lower bound, namely the largest prime divisor of $d_i$. Replacing $d_i$ by $p_i$ in the above representation of $\lambda(\alpha)$, we obtain a lower bound on it which is a subproduct of and therefore strictly bounded below by the infinite product $\prod_p{(1-\frac{1}{2^p})}$. The second inequality follows by Lemma \ref{expLem} and checking numerically that $\prod_{p\leq 17}{(1-\frac{1}{2^p})}\cdot\mathrm{exp}(-\frac{1}{2^{17}})>0.63038$, and the last inequality is clear.

For (2): The set of admissible exponents is $\{4\}\cup\mathbb{P}^+$. Replace $d_i$ divisible by an odd prime by their largest prime divisor, and should one $d_i$ happen to be a power of $2$, then by assumption, it is greater than or equal to $4$ and we may replace it by $4$. For the numerical bound, just check that $\frac{0.63038}{1-\frac{1}{2^2}}\cdot(1-\frac{1}{2^4})>0.78797$.

For (3): The set of admissible exponents is $\{6\}\cup(\mathbb{P}^+\setminus\{3\})$. Replace $d_i$ divisible by a prime distinct from $2$ and $3$ by the largest such prime. As for $d_i$ only divisible by $2$ or $3$: If there is precisely one such $d_i$, it is at least $6$ and we replace it by this number. And if one $d_i$ is a power of $2$ and another, say $d_j$, is a power of $3$, then $d_i\geq 8$ and $d_j\geq 9$. Now it is readily checked that $(1-\frac{1}{2^8})(1-\frac{1}{2^9})>1-\frac{1}{2^6}$, so for the lower bound, we can then replace both factors by $1-\frac{1}{2^6}$.

The proofs of points (4)--(8) all follow the same idea (replacing $d_i$ divisible by at least one \enquote{admissible} prime by one of them and treating the other $d_i$ separately), but are easier than the argument for point (3).
\end{proof}

\subsection{Effective classification of periodic abelian FDGs \texorpdfstring{$(G,\alpha)$}{(G,alpha)} with \texorpdfstring{$\lambda(\alpha)=\rho$}{lambda(alpha)=rho} for fixed \texorpdfstring{$\rho\geq\frac{1}{2}$}{rho>=1/2}}\label{subsec2P8}

In this subsection, we consider the following two problems:

(1) For a given $\rho\in\left[\frac{1}{2},1\right]$, determine, up to isomorphism, the complete list of periodic abelian FDGs $(G,\alpha)$ such that $\lambda(\alpha)=\rho$.

(2) For a given $\rho\in\left[\frac{1}{2},1\right]$, determine, up to isomorphism, the complete list of finite abelian groups $G$ such that $\lambda(G)=\rho$.

Note that by Theorem \ref{abelianTheo} and Theorem \ref{equalTheo}, being able to effectively solve either of these two problems is equivalent to being able to effectively solve the respective problem for all periodic FDGs or all finite groups; the only difference is that the algorithm for the more general version of (1) would, in case $\rho=\frac{1}{2}$, also add (descriptions of) the four classes of nonabelian FDGs described in Theorem \ref{equalTheo}, and the algorithm for the more general version of (2) would add (descriptions of) the underlying groups of those FDGs. We will provide a detailed algorithmic solution to (1), argue (based on the algorithm for (1)) why (2) can be effectively solved, and use the gained insights to prove the \enquote{abelian part} of Theorem \ref{equalTheo} (see Proposition \ref{abEqProp} below).

We now start to describe the algorithm for determining all solutions on the FDG level, i.e., for problem (1). Given a rational number $\rho=\frac{a}{b}\in\left[\frac{1}{2},1\right]$ such that $\mathrm{gcd}(a,b)=1$, the algorithm first determines the solutions where the order of $G$ is primary. Now it is clear by Theorem \ref{lccAbTheo}, Theorem \ref{elAbTheo} and Corollary \ref{properCor} that if $G$ is a finite abelian $2$-group having an automorphism $\alpha$ such that $\lambda(\alpha)=\frac{1}{2}$, then $(G,\alpha)$ is isomorphic to one of the following:

(1) $\mathbb{Z}/2\mathbb{Z}$ with the identity.

(2) $\mathbb{Z}/4\mathbb{Z}$ with the multiplication by $3$.

(3) $\mathbb{Z}/2\mathbb{Z}\times\mathbb{Z}/4\mathbb{Z}=\langle v_1,v_2\mid v_1^2=v_2^4=[v_1,v_2]=1\rangle$ with the automorphism given sending $v_1\mapsto v_1v_2^2,v_2\mapsto v_1v_2$ (to see that $\mathbb{Z}/2\mathbb{Z}\times\mathbb{Z}/4\mathbb{Z}$ underlies precisely one $\lambda$-maximal FDG isomorphism type, use GAP \cite{GAP4}).

(4) $(\mathbb{Z}/2\mathbb{Z})^2$ with the automorphism given by the companion matrix of $(X-1)^2\in\mathbb{F}_2[X]$.

(5) $(\mathbb{Z}/2\mathbb{Z})^3$ with the automorphism given by the companion matrix of $(X-1)^3\in\mathbb{F}_2[X]$.

Hence for $\rho=\frac{1}{2}$, the algorithm adds these FDGs to the list. On the other hand, if $\rho>\frac{1}{2}$ is to be the $\lambda$-value of some periodic abelian $p$-FDG $(G,\alpha)$, say with $|G|=p^m$, then by Corollary \ref{abStrictCor}, one of the following two cases must occur:

(A) $p$ is odd and $\rho=1-\frac{1}{p}$. In this case, by Corollary \ref{abStrictCor}, the list consists of all FDGs of the form $(\mathbb{Z}/p^m\mathbb{Z},\alpha)$ for $m\geq 1$, where $\alpha$ is the multiplication by a primitive root modulo $p^m$ (where different primitive roots yield distinct FDG isomorphism types), plus the exceptional $((\mathbb{Z}/p\mathbb{Z})^2,\alpha)$, with any $\alpha$ given, with respect to the standard basis, by $\mathrm{M}((X-g)^2)$ for some generator $g$ of $\mathbb{F}_p^{\ast}$ (where different generators correspond to distinct FDG isomorphism types).

(B) $p$ is any prime, $m\geq 2$ and $b=p^m$, whence $|G|$ can be read off from $\rho$, and the actual list of solutions is then easily determined in view of Corollary \ref{abStrictCor} (note that if $b$ is a power of $2$, one still needs to do some work, going through all additive decompositions of $\mathrm{log}_2(b)$ as described in point 1 of the classification and checking whether one \enquote{matches}).

Next, the algorithm determines the solutions where $|G|$ is not primary. Observe that by Product Lemma \ref{productLem}(2,iii), the Sylow FDGs $(G_p,\alpha_p)$ in a periodic abelian FDG $(G,\alpha)$ with $\lambda$-value $\rho\geq\frac{1}{2}$ such that $|G|$ is not primary all have $\lambda$-value greater than $\frac{1}{2}$ and hence by Corollary \ref{abStrictCor} and Product Lemma \ref{productLem}(2,ii), $G=G_2\times G_p$ for some odd prime $p$, $(G_2,\alpha_2)=\prod_{i=1}^r\mathcal{V}(P_i(X))$, with $P_i(X)\in\mathbb{F}_2[X]$ monic, primitive and irreducible for $i=1,\ldots,r$, and either $G_p$ is primary cyclic and $\Lambda(\alpha_p)=\phi(|G_p|)$, or $G_p$ is elementary abelian and $\Lambda(\alpha_p)=|G_p|-1$, or $G_p\cong(\mathbb{Z}/p\mathbb{Z})^2$ and $\Lambda(\alpha_p)=\phi(p^2)=p(p-1)$. Furthermore, by Product Lemma \ref{productLem}(2,ii) the contributions of the factors $\mathcal{V}(P_i(X))$ and of $G_p$ to the order of $\alpha$ must be pairwise coprime, and the $\lambda$-value of $\alpha$ is the product of the $\lambda$-values of the single parts. Hence if $\rho$ is to be the $\lambda$-value of some automorphism $\alpha=\alpha_2\times\alpha_p$ of some finite abelian group $G=G_2\times G_p$ whose order is not primary, then \[\rho=\frac{\prod_{i=1}^r{(2^{d_i}-1)}\cdot\Lambda(\alpha_p)}{2^{d_1+\cdots+d_r}p^m}\] for pairwise coprime $d_1,\ldots,d_r\geq 2$ such that all $2^{d_i}-1$ are coprime with $\Lambda(\alpha_p)$, and $p^m$ is the order of the Sylow $p$-subgroup $G_p$. In particular, since $b$ is a divisor of the denominator of that fraction, $b$ is divisible by at most one odd prime. In case $m\geq 2$ and $\Lambda(\alpha_p)=\phi(p^m)=p^{m-1}(p-1)$, no $2^{d_i}-1$ is divisible by $p$, so $\nu_p(b)=1$. In any case, $\nu_q(\prod_{i=1}^r{(2^{d_i}-1)}\cdot\Lambda_p)=\nu_q(a)$ for any odd prime $q$ distinct from $p$, so $a$ differs from the \enquote{original} numerator only in its $2$-adic and $p$-adic valuation. We may assume w.l.o.g.~that if $p$ divides any of the $2^{d_i}-1$, then it divides $2^{d_1}-1$. By pairwise coprimality of the $2^{d_i}-1$, the numbers $2^{d_2}-1,\ldots,2^{d_r}-1$ then are products of prime powers of the form $q^{\nu_q(a)}$ for odd primes $q\not=p$. For linguistical simplicity, we will henceforth call $n_1\in\mathbb{N}^+$ a \textit{full divisor} of $n_2\in\mathbb{N}^+$ if and only if $n_1\mid n_2$ and for all primes $p'$: If $p'\mid n_1$, then $\nu_{p'}(n_1)=\nu_{p'}(n_2)$. We denote the set of full divisors of $n\in\mathbb{N}^+$ by $\mathcal{T}(n)$ and set \[\mathcal{T}_0(n):=\{t\in\mathcal{T}(n)\mid t+1\text{ is a power of }2\}.\] The algorithm now proceeds as follows:

\begin{enumerate}
\item Compute the prime factorization of $b$.
\item If $b$ is divisible by two distinct odd primes, then there are no abelian periodic FDGs with $\lambda$-value $\rho$; output $\emptyset$ and halt.
\item If $b=2^kp^l$ for some odd prime $p$ with $l\geq 2$, then note that $p$ is the unique odd prime divisor of $|G|$ in any abelian periodic FDG $(G,\alpha)$ with $\lambda(\alpha)=\rho$. Also, by the above observations, the Sylow $p$-FDG of $(G,\alpha)$ is isomorphic to $((\mathbb{Z}/p\mathbb{Z})^m,\alpha_p)$ for some $m\geq l$ such that $\alpha_p$ is an automorphism with precisely one primary Frobenius block, which is the companion matrix of an irreducible primitive polynomial of degree $m$ over $\mathbb{F}_p$. Note that by coprimality of $a$ and $b$, we have $p\nmid a$. In order to make a case distinction according to the values of the $2^{d_i}-1$ for $i\geq 2$, compute $\mathcal{T}_0(a)$. The idea now is to loop over all tuples $(t_1,\ldots,t_r)$ with pairwise distinct and pairwise coprime entries from $\mathcal{T}_0(a)$, viewing $r$ as the number of factors $2^{d_i}-1$ of the numerator of $\lambda(\alpha_2)$, and the entries $t_1,\ldots,t_r$ as the \enquote{remnants} of the factors $2^{d_1}-1,\ldots,2^{d_r}-1$ after cancelation; more precisely, $t_1$ is interpreted as the value of $\frac{2^{d_1}-1}{p^{\nu_p(2^{d_1}-1)}}$, and $t_i$ for $i=2,\ldots,r$ as the value of $2^{d_i}-1$. Note that necessarily $\rho<\prod_{i=2}^r{(1-\frac{1}{2^{d_i}})}$, so we can refute the case if this inequality does not hold. Otherwise, if $\nu_p(2^{d_1}-1)=l'$, then $m=l+l'$, and so \[\rho=(1-\frac{1}{p^{l'}t+1})(1-\frac{1}{p^{l+l'}})\prod_{i=2}^r{(1-\frac{1}{t_i+1})},\] which leaves at most one possibility for $l'$ within an effectively bounded range of values, since the RHS of this inequality is strictly increasing in $l'$ and will eventually be strictly greater than $\rho$.
\item If $b=2^kp$ for some odd prime $p$, then again, $p$ is the unique odd prime divisor of $|G|$. If $(G_p,\alpha_p)$ is $(\mathbb{Z}/p\mathbb{Z})^m$ with $\alpha_p$ an automorphism of order $p^m-1$ as above, then the case can be treated as in point 3. However, it now may also be that the Sylow $p$-FDG of $(G,\alpha)$ is isomorphic to $(\mathbb{Z}/p^m\mathbb{Z},\alpha_p)$ for some $m\geq 2$ such that $\alpha_p$ is the multiplication by a primitive root modulo $p^m$, or to $((\mathbb{Z}/p\mathbb{Z})^2,\alpha_p)$, where the only primary Frobenius block of $\alpha_p$ is $\mathrm{M}((X-g)^2)$ for any generator $g$ of $\mathbb{F}_p^{\ast}$. If this happens, though, then we know that $p\mid\Lambda_p$, so $p$ does not divide any of the $2^{d_i}-1$, which is good news, since now \textit{all} values of the $2^{d_i}-1$ are known after fixing a distribution of full divisors of $a$ as above, meaning we then know the Frobenius type of the Sylow $2$-FDG $(G_2,\alpha_2)$ of $(G,\alpha)$. Now $1-\frac{1}{p}=\lambda(\alpha_p)=\frac{\rho}{\lambda(\alpha_2)}$, so if the leftmost term in this equality chain is not equal to the rightmost term, or if $\mathrm{gcd}(\Lambda(\alpha_2),p-1)>1$, the subcase is contradictory. Otherwise, there are two possibilities: If also $p\nmid\Lambda(\alpha_2)$, then the subcase corresponds to infinitely many isomorphism types of periodic FDGs with $\lambda$-value $\rho$, namely all of the form $(G_2,\alpha_2)\times(G_p,\alpha_p)$, where $(G_p,\alpha_p)$ is any of the finite abelian $p$-FDGs with $\lambda$-value $1-\frac{1}{p}$. And if $p\mid\Lambda(\alpha_2)$, then the subcase yields no FDGs not already contained in the first case.
\item If $b=2^k$, we are in the unfortunate position to not know $p$ right away, but at least we know, just as in point 3, that $G_p\cong(\mathbb{Z}/p\mathbb{Z})^m$ and $\Lambda(\alpha_p)=p^m-1$. Also, we now know that $2^{d_1}-1$ is divisible by $p$ (otherwise, $p^m$ in the denominator could not have canceled). We will be able to treat this point similarly to point 3, but we need to bound $p^m$ first. We do this \textit{after} fixing a distribution of full divisors $t_i$ of $a$ as before. Now note that whether or not $p$ is one of the primes \enquote{visible} in the list of prime divisors of $2^{d_1}-1$ just fixed, the precise value of $2^{d_1}-1$ equals $p^mt_1$. Therefore, \[\rho=(1-\frac{1}{p^mt_1+1})(1-\frac{1}{p^m})\prod_{i=2}^r{(1-\frac{1}{t_i+1})},\] leaving at most one possibility for $p^m$ by the same idea as at the end of point 3, and in case an odd primary number $p^m$ satisfying this equation exists, it only remains to check whether the corresponding cycle lengths from the factors are pairwise coprime.
\end{enumerate}

After this explicit description of an algorithm for problem (1), we now briefly justify the effective solvabililty of problem (2). For this, first run the algorithm for (1) to obtain all solutions on the FDG level. Of course, for each finite group $G$ and automorphism $\alpha$ of $G$, deciding whether $\lambda(G)>\lambda(\alpha)$ holds can be done algorithmically. In those cases where there are infinitely many solutions on the FDG level, it is clear by the explicit description of the algorithm above that the full FDG solution list is a union of a finite solution list with finitely many infinite lists, each of the form \[\{\prod_{i=1}^r{\mathcal{V}(P_i(X))}\times\mathcal{M}(p^m,g)\mid m\geq 2,g\text{ is a primitive root mod }p^m\}\] for some monic primitive irreducible polynomials $P_1(X),\ldots,P_r(X)\in\mathbb{F}_2[X]$. Set $d_i:=\mathrm{deg}(P_i(X))$ for $i=1,\ldots,r$, and note that for such an infinite class, either all its members are $\lambda$-maximal or all are not $\lambda$-maximal, with the check consisting in the problem of verifying whether or not there exists an automorphism of $(\mathbb{Z}/2\mathbb{Z})^{d_1+\cdots+d_r}$ whose largest cycle length $L$ is coprime with $p(p-1)$ and such that $L>\prod_{i=1}^r{(2^{d_i}-1)}$. Therefore, problem (2) indeed has an algorithmic solution.

We can now give a short \enquote{by hand} proof of the \enquote{abelian part} of Theorem \ref{equalTheo}:

\begin{propposition}\label{abEqProp}
The FDGs from points 1--6 in Theorem \ref{equalTheo} form, up to isomorphism, the complete list of periodic \textbf{abelian} FDGs with $\lambda$-value $\frac{1}{2}$.
\end{propposition}

\begin{proof}
Since the $p$-FDG case is clear and $(\mathbb{Z}/2\mathbb{Z})^2\times\mathbb{Z}/3\mathbb{Z}$ has an automorphism of order $6$ and thus of $\lambda$-value $\frac{1}{2}$, it suffices to show that up to isomorphism, there is at most one periodic FDG $(G,\alpha)$ with $\lambda(\alpha)=\frac{1}{2}$ such that $G$ is abelian and $|G|$ is not primary. Using notation from the description of the algorithm for problem (1) above, we have $a=1$, $b=2$, $l=0$ here. The complete set of full divisors of $a$ is, of course, $\{1\}$, so there is only one full divisor distribution to be considered. By the above description of the algorithm, it is clear that every distribution to be considered yields at most one solution, which concludes the proof.
\end{proof}

\section{Proofs of the main results}\label{sec3}

\subsection{Proof of Theorem \ref{abelianTheo} and Corollary \ref{strictCor}}\label{subsec3P1}

Corollary \ref{strictCor} follows immediately from Theorem \ref{abelianTheo} via Corollary \ref{abStrictCor}, so it suffices to prove Theorem \ref{abelianTheo}. Let $G$ be a finite group such that $\lambda(G)>\frac{1}{2}$. Fix an automorphism $\alpha$ of $G$ such that $\lambda(\alpha)=\lambda(G)$ and assume, for a contradiction, that $G$ is nonabelian. As observed at the beginning of Subsection \ref{subsec2P2}, setting $A:=\{1_G\}$ in Inversion Lemma \ref{inversionLem}, we find that $\mathrm{l}(G)>\frac{1}{2}$, whence by Theorem \ref{lStructureTheo}, $G$ is of one of the three types described there. We now go through those three cases in reversed order, since that is also the ordering of the cases by increasing difficulty.

If $G$ is of type (III), set $A:=\zeta G$ and let $\tilde{\alpha}$ denote the automorphism of $G/A\cong(\mathbb{Z}/2\mathbb{Z})^4$ induced by $\alpha$. By Automorphism Quotient Transfer (Lemma \ref{transferLem}(3,a)), we have $\lambda(\tilde{\alpha})>\frac{1}{2}$, which by Theorem \ref{elAbTheo} and the fact that $4$ has no additive decomposition into pairwise coprime summands greater than $1$ except for the trivial one implies that $\lambda(\tilde{\alpha})=\frac{15}{16}$. Applying Inversion Lemma \ref{inversionLem}, it follows that $\mathrm{l}(G)>\frac{15}{16}$ so that by Corollary \ref{millerCor}, $G$ is abelian, a contradiction.

If $G$ is of type (II), again set $A:=\zeta G$ and let $\tilde{\alpha}$ be the automorphism of $G/A\cong(\mathbb{Z}/2\mathbb{Z})^{2k}$ (for some $k\geq 2$) induced by $\alpha$. By Lemma \ref{anLem}(1), we have $\lambda(\tilde{\alpha})>\frac{9}{16}$, so $\mathrm{l}(G)>\frac{9}{16}$. Since by Theorem \ref{lStructureTheo}, $\mathrm{l}(G)=\frac{2^k+1}{2^{k+1}}$, it follows that $k=2$, and we derive the same contradiction as in the previous case.

If $G$ is of type (I), fix an abelian subgroup $A$ of index $2$ in $G$ and an element $x\in G\setminus A$. Then $x$ acts on $A$ by conjugation, and this action must be nontrivial, since otherwise, $G$ would be abelian. Also, this action is equal to the action of any element from $G\setminus A=xA$ since $A$ is abelian. It follows that $\zeta G\subseteq A$. Since $\lambda(\alpha)>\frac{1}{2}$, the domain of the unique largest cycle of $\alpha$ must contain both an element from $A\setminus\zeta G$ and from $xA$. Easy computations show that the centralizer of any element from $A\setminus\zeta G$ is $A$ and that the centralizer of any element from $xA$ consists of $\zeta G$ and one coset of $\zeta G$ in $xA$. Comparing the centralizer orders in both cases, we conclude that $[A:\zeta G]=2$, or $[G:\zeta G]=4$, whence $G/\zeta G\cong(\mathbb{Z}/2\mathbb{Z})^2$. Denote by $\tilde{\alpha}$ the automorphism of $G/\zeta G$ induced by $\alpha$. Then in view of $\lambda(\tilde{\alpha})>\frac{1}{2}$, we conclude that $\lambda(\tilde{\alpha})=\frac{3}{4}$, and thus, by an application of Inversion Lemma \ref{inversionLem}, $\mathrm{l}(G)>\frac{3}{4}$, so one last time, by Corollary \ref{millerCor}, we arrive at the contradictory conclusion that $G$ is abelian.\qed

\subsection{Proof of Theorem \ref{equalTheo}}\label{subsec3P2}

By Proposition \ref{abEqProp}, it suffices to show that any periodic FDG $(G,\alpha)$ such that $\lambda(\alpha)=\frac{1}{2}$ and $G$ is nonabelian is isomorphic to one of the FDGs from points 7-10.

The proof is structured into the following steps:

\begin{enumerate}
\item $G$ is of type (I), i.e., $G$ has an abelian subgroup $A$ of index $2$.

\item $A$ can be chosen to be $\alpha$-admissible, which clearly implies that the domain of the unique largest cycle $\sigma$ of $\alpha$ coincides with $G\setminus A$.

\item $(G,\alpha)$ is isomorphic to one of the nonabelian FDGs from the theorem, and the dependence of the isomorphism type on the case-specific parameter is as asserted by the theorem.
\end{enumerate}

For 1: We can derive a contradiction in the cases where $G$ is of type (II) or of type (III) as in the proof of Theorem \ref{abelianTheo} by observing that by Corollary \ref{properCor}, for the induced automorphism $\tilde{\alpha}$ on the central quotient of $G$, $\lambda(\tilde{\alpha})>\frac{1}{2}$.

For 2: Assume that $\mathrm{dom}(\sigma)\not=G\setminus A$. Then $\mathrm{dom}(\sigma)$ intersects both with $A\setminus\zeta G$ and with $G\setminus A$, so as in the proof of Theorem \ref{abelianTheo}, we conclude that $[G:\zeta G]=4$; in particular, $G$ is nilpotent of class $2$ with all Sylow subgroups except for the Sylow $2$-subgroup being abelian. If $|G|$ is divisible by some odd prime $p$, then by Product Lemma \ref{productLem}(2,iii), $\lambda(G_2)>\frac{1}{2}$ for the Sylow $2$-subgroup $G_2$ of $G$, which by Theorem \ref{abelianTheo} implies that $G_2$ and hence $G$ is abelian, a contradiction. Therefore, $G$ is a $2$-group. Furthermore, the automorphism $\tilde{\alpha}$ of $G/\zeta G$ induced by $\alpha$ has $\lambda$-value at least $\frac{1}{2}$ by the Automorphism Quotient Transfer Lemma \ref{transferLem}(3,a), but if $\lambda(\tilde{\alpha})>\frac{1}{2}$, we arrive at the same contradiction as in the proof of Theorem \ref{abelianTheo}. Hence $\lambda(\tilde{\alpha})=\frac{1}{2}$, and $\mathrm{dom}(\sigma)$ is equal to the union of two cosets of $\zeta G$ in $G$, one of them being $A\setminus\zeta G$. Setting $\beta:=\alpha^2$, we find that $\beta$ restricts to an automorphism $\overline{\beta}$ of $A$ such that $\lambda(\overline{\beta})=\frac{1}{2}$. Since $A$ is a finite abelian $2$-group, by Proposition \ref{abEqProp} and nonabelianity of $G$, $(A,\overline{\beta})$ is isomorphic to one of the following:

(a) $(\mathbb{Z}/2\mathbb{Z})^2$ with the automorphism represented by the companion matrix of $(X-1)^2\in\mathbb{F}_2[X]$,

(b) $\mathbb{Z}/4\mathbb{Z}$ with multiplication by $3$,

(c) $(\mathbb{Z}/2\mathbb{Z})^3$ with the automorphism represented by the companion matrix of $(X-1)^3\in\mathbb{F}_2[X]$,

(d) $\mathbb{Z}/2\mathbb{Z}\times\mathbb{Z}/4\mathbb{Z}$ with the automorphism represented by the matrix $\begin{pmatrix}1 & 1 \\ 2 & 1\end{pmatrix}$.

In cases (a) and (b), we find that $|G|=8$, whence $G\cong\mathrm{D}_8$ or $G\cong\mathrm{Dic}_8=\mathrm{Q}_8$. $\mathrm{D}_8$ has a characteristic abelian subgroup of index $2$, by which we can replace $A$. $\mathrm{Q}_8$, in turn, does not have a characteristic subgroup of index $2$, but it is not difficult to see that in this case, $\alpha$, being an automorphism of order $4$ of $\mathrm{Q}_8$, has a fixed point $f$ of order $4$, so we may replace $A$ by $\langle f\rangle$.

In cases (c) and (d), we conclude that $|G|=16$ and that $\mathrm{ord}(\alpha)=8$, since the elements of $G$ whose cycle under $\alpha$ has length a divisor of $8$ form a subgroup of $G$. Observe that $|\zeta G|=4$, so the restriction of $\alpha$ to $\zeta G$ only has cycles of length at most $2$, whence $\overline{\beta}$ has, apart from its $4$-cycle, four fixed points. However, it is readily checked that each of the automorphisms as described in (c) and (d) only has two fixed points, a contradiction.

For 3: Assuming w.l.o.g.~that $\mathrm{dom}(\sigma)=G\setminus A$, which is a coset of $A$ in $G$, we can conclude by Transfer Lemma \ref{transferLem}(2) that, denoting by $\overline{\alpha}$ the restriction of $\alpha$ to $A$, fixing $x\in G\setminus A$ and letting $a_0\in A$ be such that $\alpha(x)=xa_0$, we have $\lambda(\mathrm{A}_{a_0,\overline{\alpha}})=1$, whence we can derive some heavy restrictions on both $A$ and $\alpha$ by Corollary \ref{affineCor}. There are two cases to distinguish, according to whether or not $A$ is cylic:

If $A$ is cyclic, then $G$ has a presentation of the form $\langle r,x\mid r^n=1,x^2=r^k,xrx^{-1}=r^l\rangle$ for some $n\geq 3$ and $k,l\in\mathbb{N}$, where $A=\langle r\rangle$. We find that $r^{l^2}=x^2\cdot r\cdot x^{-2}=r^k\cdot r\cdot r^{-k}=r$, whence \begin{equation}\label{eq1} l^2\equiv 1\hspace{3pt}(\mathrm{mod}\hspace{3pt}n).\end{equation} Also, $r^k=x^2=x\cdot x^2\cdot x^{-1}=x\cdot r^k\cdot x^{-1}=r^{lk}$, or equivalently \begin{equation}\label{eq2} k(l-1)\equiv 0\hspace{3pt}(\mathrm{mod}\hspace{3pt}n).\end{equation} Note that equations (\ref{eq1}) and (\ref{eq2}) impose restrictions purely on the structure of $G$, and not on $\alpha$. Now by assumption, there exist $m,a\in\mathbb{N}$ such that $\alpha(r)=r^m$, $\alpha(x)=xr^a$, and $y\mapsto my+a$ is an affine map of $\mathbb{Z}/n\mathbb{Z}$ with $\lambda$-value $1$, which by Corollary \ref{affineCor} lets us infer the following restrictions on $m$ and $a$: \begin{equation}\label{eq3}\mathrm{gcd}(a,n)=1\end{equation} and \begin{equation}\label{eq4}p\text{ prime and }p\mid n\Rightarrow m\equiv 1\hspace{3pt}(\mathrm{mod}\hspace{3pt}p),\text{ and }4\mid n\Rightarrow m\equiv 1\hspace{3pt}(\mathrm{mod}\hspace{3pt}4).\end{equation} Additionally, $a$ and $m$ must be chosen such that the three defining relations are respected. It is, however, not difficult to see that the first and the third of them are respected for any choice of $a$ and $m$, so only the second defining relation yields one further restriction, namely, as an easy computation shows, \begin{equation}\label{eq5}k(m-1)\equiv a(l+1)\hspace{3pt}(\mathrm{mod}\hspace{3pt}n).\end{equation}

Consider an odd prime divisor $p$ of $n$ and set $k_p:=\nu_p(n)$. Then we can infer from Equation (\ref{eq1}) that $l\equiv\pm 1\hspace{3pt}(\mathrm{mod}\hspace{3pt}p^{k_p})$. However, if $l\equiv 1\hspace{3pt}(\mathrm{mod}\hspace{3pt}p^{k_p})$, we get a contradiction, since Equation (\ref{eq5}) in view of Equations (\ref{eq3}) and (\ref{eq4}) then implies that $0\equiv k(m-1)\equiv 2a\not\equiv 0\hspace{3pt}(\mathrm{mod}\hspace{3pt}p)$. Hence we conclude that \begin{equation}\label{eq6}l\equiv -1\hspace{3pt}(\mathrm{mod}\hspace{3pt}p^{k_p}),\end{equation} which in view of Equation (\ref{eq2}) implies that \begin{equation}\label{eq7}k\equiv 0\hspace{3pt}(\mathrm{mod}\hspace{3pt}p^{k_p}).\end{equation}

If $k_2:=\nu_2(n)>0$, Equation (\ref{eq1}) yields that $l\hspace{3pt}\mathrm{mod}\hspace{3pt}2^{k_2}\in\{1,2^{k_2-1}-1,2^{k_2-1}+1,2^{k_2}-1\}$. We will show that actually \begin{equation}\label{eq8}l\equiv -1\hspace{3pt}(\mathrm{mod}\hspace{3pt}2^{k_2}),\end{equation} so that in view of Equation (\ref{eq6}), the conjugation action of $x$ on $\langle r\rangle$ is by inversion. Note that Equation (\ref{eq8}) is clear if $k_2=1$, so we may assume that $k_2>1$, whence by Equation (\ref{eq4}), $m\equiv 1\hspace{3pt}(\mathrm{mod}\hspace{3pt}4)$. If $l\equiv 1\hspace{3pt}(\mathrm{mod}\hspace{3pt}2^{k_2})$, then by Equation (\ref{eq5}), we have $0\equiv k\cdot\frac{m-1}{2}\equiv a\not\equiv 0\hspace{3pt}(\mathrm{mod}\hspace{3pt}2^{k_2-1})$, a contradiction. If $l\equiv 2^{k_2-1}-1\hspace{3pt}(\mathrm{mod}\hspace{3pt}2^{k_2})$, then by Equation (\ref{eq2}), we find that $k\equiv 0\hspace{3pt}(\mathrm{mod}\hspace{3pt}2^{k_2-1})$, whence by Equation (\ref{eq5}), $0\equiv \frac{k}{2^{k_2-1}}(m-1)\equiv a\equiv 1\hspace{3pt}(\mathrm{mod}\hspace{3pt}2)$, a contradiction. Finally, if $l\equiv 2^{k_2-1}+1\hspace{3pt}(\mathrm{mod}\hspace{3pt}2^{k_2})$, then by Equation (\ref{eq2}), $k\equiv 0\hspace{3pt}(\mathrm{mod}\hspace{3pt}2)$, which by Equation (\ref{eq5}) yields the contradictory $0\equiv \frac{k}{2}\cdot(m-1)\equiv a\cdot(2^{k_2-2}+1)\equiv 1\hspace{3pt}(\mathrm{mod}\hspace{3pt}2)$. This concludes the proof of Equation (\ref{eq8}).

Now by Equations (\ref{eq2}) and (\ref{eq8}), we conclude that $k\equiv 0\hspace{3pt}(\mathrm{mod}\hspace{3pt}2^{k_2-1})$, whence either $k\equiv 0\hspace{3pt}(\mathrm{mod}\hspace{3pt}2^{k_2})$, in which case $G\cong\mathrm{D}_{2n}$, or $k\equiv 2^{k_2-1}\hspace{3pt}(\mathrm{mod}\hspace{3pt}2^{k_2})$, i.e., $G\cong\mathrm{Dic}_{2n}$. In both cases, we find that choosing $a$ and $m$ according to Equations (\ref{eq3}) and (\ref{eq4}), Equation (\ref{eq5}) is automatically satisfied because of Equations (\ref{eq6})--(\ref{eq8}), so any such choice of $a$ and $m$ leads to an automorphism $\alpha$ of $\lambda$-value $\frac{1}{2}$. Also, it is not difficult to see that $m\hspace{3pt}\mathrm{mod}\hspace{3pt}n$ is determined by the FDG isomorphism type of $(G,\alpha)$ in both cases, since $A$ is characteristic in $G$ except when $G\cong\mathrm{Dic}_8$, in which case necessarily $m=1$. Hence in order to conclude the proof of the case where $A$ is cyclic, it suffices to show that keeping $m$ fixed and replacing $a$ by $1$ does not change the FDG isomorphism type. However, it is readily checked that conjugation by the automorphism sending $r\mapsto r^{a^{-1}},x\mapsto x$ (with $a^{-1}\in Z_n$ the multiplicative inverse of $a$ modulo $n$) transforms $\alpha$ into the automorphism sending $r\mapsto r^m,x\mapsto xr$.

If $A$ is not cyclic, then by Corollary \ref{affineCor}, $A$ is isomorphic to $(\mathbb{Z}/2\mathbb{Z})^2\times\mathbb{Z}/o\mathbb{Z}$ for some odd $o\geq 1$. Any element $x\in G\setminus A$ acts on $A$ by an automorphism of order $2$, and hence on the Sylow $2$-subgroup of $A$, $(\mathbb{Z}/2\mathbb{Z})^2$, by the identity or by an automorphism of order $2$. We will first show that it must act by the identity. Otherwise, there exist two nontrivial elements $r_1,r_2\in(\mathbb{Z}/2\mathbb{Z})^2$ swapped in a transposition by conjugation with $x$, from which we conclude that there exist $\epsilon_1,\epsilon_2\in\{0,1\}$ and $k,l\in\mathbb{N}$ with $\mathrm{gcd}(l,o)=1$ such that $G=\langle r_1,r_2,r,x\mid r_1^2=r_2^2=r^o=[r_1,r_2]=[r_1,r]=[r_2,r]=1,x^2=r_1^{\epsilon_1}r_2^{\epsilon_2}r^k,xr_1x^{-1}=r_2,xr_2x^{-1}=r_1,xrx^{-1}=r^l\rangle$. Computing $x\cdot x^2\cdot x^{-1}$ in two ways, we deduce that $\epsilon_1=\epsilon_2=:\epsilon$. Fix $a,m\in\mathbb{N}$ and $f_1,f_2\in\{0,1\}$ such that $\alpha(r)=r^m,\alpha(x)=xr_1^{f_1}r_2^{f_2}r^a$, and note that $\alpha$ restricts to an automorphism of $\langle r_1,r_2\rangle$ of order $2$ (by Corollary \ref{affineCor}). Hence either $\alpha$ swaps $r_1$ and $r_2$ in a transposition, or w.l.o.g.~$\alpha(r_1)=r_1,\alpha(r_2)=r_1r_2$, and both cases are contradictory: If $\alpha(r_1)=r_2,\alpha(r_2)=r_1$, then since by Corollary \ref{affineCor}, $r_1^{f_1}r_2^{f_2}$ must not be a fixed point of $\alpha$, we can assume w.l.o.g.~(swapping the indices of $r_1$ and $r_2$ if necessary) that $(f_1,f_2)=(1,0)$. Considering that $\alpha$ must respect the defining relation $x^2=r_1^{\epsilon}r_2^{\epsilon}r^k$, we obtain that $r_1^{1+\epsilon}r_2^{1+\epsilon}=r_1^{\epsilon}r_2^{\epsilon}$, a contradiction. If $\alpha(r_1)=r_1,\alpha(r_2)=r_1r_2$, then it is easy to see that $\alpha$ cannot respect the defining relation $xr_1x^{-1}=r_2$.

By now, we know that $G$ has a presentation of the form $G=\langle r_1,r_2,r,x\mid r_1^2=r_2^2=r^o=[r_1,r_2]=[r_1,r]=[r_2,r]=1,x^2=r_1^{\epsilon_1}r_2^{\epsilon_2}r^k,[x,r_1]=1,[x,r_2]=1,xrx^{-1}=r^l\rangle$ for some $\epsilon_1,\epsilon_2\in\{0,1\}$ and $k,l\in\mathbb{N}$ such that $o>1$ (since otherwise, $G$ is abelian), $\mathrm{gcd}(l,o)=1$, and w.l.o.g., we may assume that $\alpha(r_1)=r_2,\alpha(r_2)=r_1,\alpha(x)=xr_1r^a,\alpha(r)=r^m$ for some $a,m\in\mathbb{N}$ with \begin{equation}\label{eq10}\mathrm{gcd}(a,o)=1\end{equation} and \begin{equation}\label{eq11}p\text{ prime and }p\mid o \Rightarrow m\equiv 1\hspace{3pt}(\mathrm{mod}\hspace{3pt}p).\end{equation} Computing, as before, $x\cdot x^2\cdot x^{-1}$ in two ways, we obtain $k(l-1)\equiv 0\hspace{3pt}(\mathrm{mod}\hspace{3pt}o)$. Also, we find again that only the defining relation $x^2=r_1^{\epsilon_1}r_2^{\epsilon_2}r^k$ imposes further restrictions, namely $\epsilon_1=\epsilon_2=:\epsilon$ and $k(m-1)\equiv a(l+1)\hspace{3pt}(\mathrm{mod}\hspace{3pt}o)$. Hence just as before, we can conclude that $l\equiv -1\hspace{3pt}(\mathrm{mod}\hspace{3pt}o)$, whence the conjugation action of $x$ on $A$ is by inversion, and $k\equiv 0\hspace{3pt}(\mathrm{mod}\hspace{3pt}o)$. Conversely, for any choice of $\epsilon\in\{0,1\}$ and of $a,m$ as in Equations (\ref{eq10}) and (\ref{eq11}), we obtain a periodic FDG $(G,\alpha)$ with $\lambda$-value $\frac{1}{2}$. It is readily verified that for both values of $\epsilon$, the map $r_1\mapsto r_1,r_2\mapsto r_2,r\mapsto r^{a^{-1}},x\mapsto x$ extends to an automorphism $\beta$ of $G$ such that $\beta\circ\alpha\circ\beta^{-1}$ maps $r_1\mapsto r_2,r_2\mapsto r_1,r\mapsto r^m,x\mapsto xr_1r$, so without changing the FDG isomorphism type, we may assume that $a=1$. Since the order of $G$ depends injectively on $o$, it remains to check that for fixed $o$ and $\epsilon$ (and $a=1$), different choices of $m\in Z_o$ such that (\ref{eq11}) holds lead to nonisomorphic FDGs. To see this, note that since $[G:\zeta G]=2o>2$, by the above observations on centralizer orders of elements of $A\setminus\zeta G$ and $G\setminus A$, the subgroup $A$, and hence also $\langle r\rangle$, is characteristic in $G$.\qed

\subsection{Proof of Theorem \ref{denseTheo}}\label{subsec3P3}

First, observe that the beginnings of the decimal digit expansions of $\rho_0$ and $\rho_1$ are as specified, since on the one hand, by Lemma \ref{expLem}, $\rho_0>\frac{4}{5}\cdot\prod_{2\leq p\leq 37}{(1-\frac{1}{2^p})}\cdot\mathrm{exp}(-\frac{1}{2^{39}})>0.504307524$ as well as $\rho_1>\frac{8}{9}\cdot\prod_{n=3,5,8}{(1-\frac{1}{2^n})}\cdot\prod_{11\leq p\leq 37}{(1-\frac{1}{2^p})}\cdot\mathrm{exp}(-\frac{1}{2^{39}})>0.750063685$, and on the other hand, $\rho_0<\frac{4}{5}\cdot\prod_{2\leq p\leq 37}{(1-\frac{1}{2^p})}<0.504307525$  and $\rho_1<\frac{8}{9}\cdot\prod_{n=3,5,8}{(1-\frac{1}{2^n})}\cdot\prod_{11\leq p\leq 37}{(1-\frac{1}{2^p})}<0.750063686$.

For (1): Let $\rho\in\mathrm{im}(\lambda)$ and fix a periodic FDG $(G,\alpha)$ such that $\lambda(\alpha)=\rho$. Set $L:=\Lambda(\alpha)$ and let $p_1,\ldots,p_r$ denote the odd prime divisors of $L$. Define $o$ as the least common multiple of the multiplicative orders of $2$ modulo the $p_i$, fix, for $n\in\mathbb{N}^+$, a monic primitive irreducible polynomial $P_n(X)\in\mathbb{F}_2[X]$ of degree $n\cdot o+1$ and set $(H_n,\beta_n):=\mathcal{V}(P_n(X))$ and $(G_n,\alpha_n):=(G,\alpha)\times (H_n,\beta_n)$. By construction, $\Lambda(\beta_n)=2^{n\cdot o+1}-1\equiv 1\hspace{3pt}(\mathrm{mod}\hspace{3pt}p_i)$ for $i=1,\ldots,r$, so that in particular, $\mathrm{gcd}(\Lambda(\alpha),\Lambda(\beta_n))=1$. Therefore, by Product Lemma \ref{productLem}(2,ii), $\lambda(\alpha_n)=\lambda(\alpha)\cdot\lambda(\beta_n)=\rho\cdot(1-\frac{1}{2^{n\cdot o+1}})$, which converges to $\rho$ from below as $n\to\infty$.

For (2): Assume, for a contradiction, that there exists a periodic FDG $(G,\alpha)$ such that $\lambda(\alpha)\in\left(\frac{1}{2},\rho_0\right]$. Then by Corollary \ref{strictCor}, Lemma \ref{anLem}(1) and the fact that the interval $\left(\frac{1}{2},\rho_0\right]$ does not contain any numbers of the form $1-\frac{1}{p^m}$ for odd primes $p$ and $m\in\mathbb{N}^+$, we conclude that $|G|$ is not primary, and $(G,\alpha)$ decomposes as $(G,\alpha)=(G_2,\alpha_2)\times(G_p,\alpha_p)$, where $G_2$ is the Sylow $2$-subgroup of $G$, $p$ the unique odd prime divisor of $|G|$, and $G_p$ the Sylow $p$-subgroup of $G$. The crucial observation now is the following: Again by Lemma \ref{anLem}(1), $\lambda(\alpha_2)>0.63038>\rho_0$, and the potential values of $\lambda(\alpha_p)$, i.e., by Corollary \ref{strictCor}, the numbers $1-\frac{1}{p^m}$ for odd $p$ and $m\geq 1$, converge to $1$ for $p^m\to\infty$. Hence there are only finitely many possibilities for $\lambda(\alpha_p)$; more precisely, numerical computations show that $\frac{6}{7}\cdot 0.63038>0.54>\rho_1$, whence $\lambda(\alpha_p)\in\{\frac{2}{3},\frac{4}{5}\}$. It is clear by Lemma \ref{anLem}(1) and the assumption $\lambda(\alpha)>\frac{1}{2}$ that $\lambda(\alpha)>\rho_0$ if $\lambda(\alpha_p)=\frac{4}{5}$, whence we conclude that $\lambda(\alpha_p)=\frac{2}{3}$. Let $(d_1,\ldots,d_r)$ be the Frobenius type of $(G_2,\alpha_2)$. Then $d_1>2$, since otherwise $\lambda(\alpha)\leq\frac{2}{3}\cdot(1-\frac{1}{2^2})=\frac{1}{2}$, a contradiction. Therefore, by Lemma \ref{anLem}(2), we find that $\lambda(\alpha)>\frac{2}{3}\cdot(1-\frac{1}{2^4})\cdot\prod_{p\geq 3}{(1-\frac{1}{2^p})}>\frac{2}{3}\cdot\prod_{n=3,4,5}{(1-\frac{1}{2^n})}\cdot\mathrm{exp}(-\frac{1}{2^5})>0.51$, another contradiction. The second part of the claim follows by considering the following sequence $((G_n,\alpha_n))_{n\geq 0}$ of periodic FDGs: Set $(G,\alpha):=\mathcal{M}(5,2)$, choose, for $i\in\mathbb{N}^+$, a monic primitive irreducible $P_i(X)\in\mathbb{F}_2[X]$ of degree $p_i$, the $i$-th prime (starting with $p_1=2$), and set $(G_n,\alpha_n):=(G,\alpha)\times\prod_{i=1}^n{\mathcal{V}(P_i(X))}$. Then $\lambda(\alpha_n)=\frac{4}{5}\cdot\prod_{i=1}^n{(1-\frac{1}{2^{p_i}})}$, which converges to $\rho_0$ for $n\to\infty$.

For (3): The proof of the second part is analogous to the one of the second part of (2). As for the first part: This is also similar to (2), but more involved. Just as before, we can show that $|G|$ is not primary, and write $(G,\alpha)=(G_2,\alpha_2)\times(G_p,\alpha_p)$. Invoking Lemma \ref{anLem}(2) this time and observing that $\frac{22}{23}\cdot 0.78797>0.753>\rho_1$, we conclude that $\lambda(\alpha_p)\in\{\frac{4}{5},\frac{6}{7},\frac{8}{9},\frac{10}{11},\frac{12}{13},\frac{16}{17},\frac{18}{19}\}$. We will now derive a contradiction in each of the seven cases. For all cases, we assume that $r\in\mathbb{N}^+$ and $d_1,\ldots,d_r\in\mathbb{N}^+\setminus\{1,2\}$ are such that $(d_1,\ldots,d_r)$ is the Frobenius type of $\alpha_2$.

\begin{enumerate}
\item If $\lambda(\alpha_p)=\frac{4}{5}$, then $d_1\geq 5$, since otherwise $\lambda(\alpha)\leq(1-\frac{1}{2^4})\cdot\frac{4}{5}=0.75$. But by Lemma \ref{anLem}(3), this implies that $\lambda(\alpha)>\frac{4}{5}\cdot\prod_{n=5,6,7}{(1-\frac{1}{2^n})}\cdot\mathrm{exp}(-\frac{1}{2^9})>0.755$, a contradiction.

\item If $\lambda(\alpha_p)=\frac{6}{7}$, then $d_1\geq 4$, since otherwise $\lambda(\alpha)\leq(1-\frac{1}{2^3})\cdot\frac{6}{7}=0.75$. However, Lemma \ref{anLem}(4) now yields $\lambda(\alpha)>\frac{6}{7}\cdot\prod_{n=4,5,7,9}{(1-\frac{1}{2^n})}\cdot\mathrm{exp}(-\frac{1}{2^9})>0.76$, a contradiction.

\item If $\lambda(\alpha_p)=\frac{8}{9}$, then by the second half of the argument in the previous case, we conclude \textit{a fortiori} that $d_1=3$. Since $(1-\frac{1}{2^3})\cdot\frac{8}{9}=\frac{7}{9}>\rho_0$, we must have $r\geq 2$. Also, $d_2\geq 5$, since otherwise $\lambda(\alpha)\leq\frac{8}{9}\cdot(1-\frac{1}{2^3})(1-\frac{1}{2^4})<0.73$. However, if $d_2\geq 6$, then by Lemma \ref{anLem}(5), we would have $\lambda(\alpha)>\frac{8}{9}\cdot\prod_{n=3,7,8}{(1-\frac{1}{2^n})}\cdot\mathrm{exp}(-\frac{1}{2^9})>0.76$. Hence $d_2=5$, and since $\frac{8}{9}\cdot(1-\frac{1}{2^3})(1-\frac{1}{2^5})>0.753$, we must have $r\geq 3$. Since $6$ is not coprime with the exponent $3$ already \enquote{in use}, we have $d_3\geq 7$. However, if $d_3=7$, then $\lambda(\alpha)\leq\frac{8}{9}\cdot\prod_{n=3,5,7}{(1-\frac{1}{2^n})}<0.748$. Hence $d_3>7$, and by Lemma \ref{anLem}(6), we can now conclude that $\lambda(\alpha)>\rho_1$, a contradiction.

\item If $\lambda(\alpha_p)=\frac{10}{11}$, then \textit{a fortiori}, $d_1=3$ and $r\geq 2$. If $d_2>4$, then by Lemma \ref{anLem}(7), we would have $\lambda(\alpha)>\frac{10}{11}\cdot\prod_{n=3,5,7,8}{(1-\frac{1}{2^n})}\cdot\mathrm{exp}(-\frac{1}{2^9})>0.76$. Hence $d_2=4$, which implies that $\lambda(\alpha)\leq\frac{10}{11}\cdot(1-\frac{1}{2^3})(1-\frac{1}{2^4})<0.746$, a contradiction.

\item If $\lambda(\alpha_p)=\frac{12}{13}$, then \textit{a fortiori}, $r\geq 2$ and $d_1=3,d_2=4$. Since $\frac{12}{13}\cdot(1-\frac{1}{2^3})(1-\frac{1}{2^4})>0.757$, we conclude that $r\geq 3$. If $d_3>5$, then by Lemma \ref{anLem}(8), we would get $\lambda(\alpha)>\frac{12}{13}\cdot\prod_{n=3,4,7,11}{(1-\frac{1}{2^n})}\cdot\mathrm{exp}(-\frac{1}{2^{11}})>0.7505$. Hence $d_3=5$, which yields $\lambda(\alpha)\leq\frac{12}{13}\cdot\prod_{n=3,4,5}{(1-\frac{1}{2^n})}<0.74$, and thus a contradiction.

\item If $\lambda(\alpha_p)=\frac{16}{17}$, then \textit{a fortiori}, $r\geq 3$ and $d_1=3,d_2=4,d_3=5$. We conclude that $\lambda(\alpha)\leq\frac{16}{17}\cdot\prod_{n=3,4,5}{(1-\frac{1}{2^n})}<0.748$, a contradiction.

\item Finally, if $\lambda(\alpha_p)=\frac{18}{19}$, then \textit{a fortiori}, $r\geq 3$ and $d_1=3,d_2=4,d_3=5$. Since $\frac{18}{19}\cdot\prod_{n=3,4,5}{(1-\frac{1}{2^n})}>0.752$, we conclude that $r\geq 4$. Also, $d_4\geq 7$. However, if $d_4=7$, then $\lambda(\alpha)\leq\frac{18}{19}\cdot\prod_{n=3,4,5,7}{(1-\frac{1}{2^n})}<0.747$, a contradiction. Hence $d_4\geq 11$, and by an application of Lemma \ref{expLem}, we conclude that $\lambda(\alpha)>\frac{18}{19}\cdot\prod_{n=3,4,5}{(1-\frac{1}{2^n})}\cdot\mathrm{exp}(-\frac{1}{2^9})>0.751$, the final contradiction.\qed
\end{enumerate}

\section{Concluding remarks}\label{sec4}

\subsection{A remark on pseudorandom number generation}\label{subsec4P1}

Both cyclic and elementary abelian groups play an important role in pseudorandom number generation. Let us quickly explain why, which will also allow us to point out a crosslink between our results and pseudorandom number generation. L'Ecuyer in \cite{Lec94a} (see also \cite{Lec12a}) gives the following definition:

\begin{deffinition}\label{rngDef}
A \textbf{random number generator} (\textbf{RNG}) is a quintuple $(S,\mu,f,U,g)$, where $S$ is a finite set of \textbf{states}, $\mu$ is a probability distribution on $S$ to choose the \textbf{seed} $s_0$, $f:S\rightarrow S$ is the \textbf{transition function}, and $g:S\rightarrow U$ is the \textbf{output function}.
\end{deffinition}

By this definition, any RNG has an FDS as underlying structure. An RNG generates a sequence of numbers by iterative application of the transition function $f$ to the seed $s_0$ and \enquote{translating} the resulting sequence of states into a number sequence by application of $g$. The sequence, although not truly random, is required to satisfy certain distribution properties making it \enquote{look random}. One basic requirement is for the underlying FDS $(S,f)$ to have large orbits (see Definition \ref{fdsDef}), that is, large cycles if $f$ is a permutation of $S$. On the other hand, for practical reasons, one wants to be able to compute $f(x)$ for all $x\in S$ efficiently and store $f$ in an economic way (so definition of $f$ by a table of values is unacceptable). In view of this, a natural choice for $(S,f)$ is a gFDG (Definition \ref{affineDef}) $(G,A)$, $A=\mathrm{A}_{g_0,\varphi}$: For storing $A$, one only needs to store $g_0$ and, for a fixed minimal generating set $X$ of $G$, the values $\varphi(x)$ for $x\in X$ (and by Lagrange's theorem, we are guaranteed to have $|X|\leq\mathrm{log}_2(|G|)$). Efficient computation of $A(g)$ for all $g\in G$ is ensured provided efficient computation with normal forms in $G$. Indeed, two of the classical types of RNGs noted, for example, by Niederreiter \cite[pp.~168ff. and 205ff.]{Nie92a}, are of this form:

(1) Linear congruential generators based on iteration of affine maps in finite cyclic groups.

(2) Pseudorandom vector generators based on iteration of matrix transformations in powers of finite cyclic groups.

As we have already seen in Example \ref{lambdaEx}, both cases contain FDSs with large cycles, and more refined distribution properties of RNGs associated with such FDSs are also well-studied, see, for instance, \cite[Theorems 7.3, 7.4, 10.4 and 10.6]{Nie92a}. What strikes the eye is that the underlying groups in (1) and (2) are very \enquote{basic} ones, and one is tempted to ask:

\begin{quesstion}\label{rngQues}
Which gFDGs, apart from the classical examples, may be used for pseudorandom number generation?
\end{quesstion}

Note that an answer to this question depends, of course, on the quality criteria for RNGs which one has in mind. We will not go into any more detail here (see, however, \cite[Section 7.2, pp.~166--168]{Nie92a}), but we note that if the threshold for \enquote{acceptable} cycle structures is set such that any choice of $(G,A)$ with $\lambda(A)<\frac{1}{2}$ is not acceptable for pseudorandom number generation, our results Corollary \ref{strictCor} and Theorem \ref{equalTheo} show that if $A$ is assumed to be an automorphism of $G$, there are no \enquote{fundamentally new} examples of such groups (note that the underlying FDSs of the nonabelian FDGs from points 7--10 in Theorem \ref{equalTheo} all are disjoint unions of the underlying FDSs of two linear congruential generators).

\subsection{Outlook}\label{subsec4P2}

Two questions naturally arise in view of the results of this paper:

\begin{quesstion}\label{ques1}
Can we prove nontrivial results on groups satisfying the automorphism $\rho$-LCC for some $\rho<\frac{1}{2}$ as well?
\end{quesstion}

\begin{quesstion}\label{ques2}
Can we prove any nontrivial results on finite groups satisfying an affine LCC in general?
\end{quesstion}

As for Question \ref{ques2}, note that in contrast to the situation for automorphism LCCs, even finite groups $G$ with $\lambda_{\mathrm{aff}}(G)=1$ need not be abelian. For example, it is easy to see that for $G=\mathrm{D_{2n}}=\langle r,x\mid r^n=x^2=1,xrx^{-1}=r^{-1}\rangle$, the periodic affine map $\mathrm{A}_{x,\alpha}$, with $\alpha$ defined by $\alpha(r)=r,\alpha(x)=xr$, has $\lambda$-value $1$. As we already observed after Definition \ref{affineDef}, for any $\rho\in\left(0,1\right)$, the affine $\rho$-LCC is a weaker condition than the automorphism $\rho$-LCC, so being able to prove something nontrivial about groups satisfying the affine $\rho$-LCC for some $\rho<\frac{1}{2}$ would in particular answer Question \ref{ques1} in the affirmative. We will address these questions in a subsequent paper.

\section{Acknowledgements}

The author would like to thank Peter Hellekalek for his many helpful comments.


\begin{thebibliography}{1}

\bibitem{DM89a}
M.~Deaconescu and D.~MacHale, Odd order groups with an automorphism cubing many elements, {\em J. Austral. Math. Soc. Ser. A} \textbf{46}(2):281--288, 1989.

\bibitem{DF04a}
D.S.~Dummit and R.M.~Foote, {\em Abstract Algebra}, John Wiley \& Sons, Inc., Hoboken, NJ, 3rd ed.~2004.

\bibitem{Els59a}
B.~Elspas, The theory of autonomous linear sequential networks, {\em IRE Trans. Circuit Theory} \textbf{CT-6}:45--60, 1959.

\bibitem{GAP4}
The GAP~Group, \emph{GAP -- Groups, Algorithms, and Programming, Version 4.7.5}, 2014, \url{http://www.gap-system.org}.

\bibitem{GMPS15a}
S.~Guest, J.~Morris, C.E.~Praeger and P.~Spiga, On the maximum orders of elements of finite almost simple groups and primitive permutation groups, {\em Trans. Amer. Math. Soc.}, to appear, arXiv:1301.5166 [math.GR].

\bibitem{Heg05a}
P.V.~Hegarty, Soluble groups with an automorphism inverting many elements, {\em Math. Proc. R. Ir. Acad.} \textbf{105A}(1):59--73, 2005.

\bibitem{Her05a}
R.A.~Hern{\'a}ndez-Toledo, Linear finite dynamical systems, {\em Comm. Algebra} \textbf{33}(9):2977--2989, 2005.

\bibitem{HR07a}
C.J.~Hillar and D.L.~Rhea, Automorphisms of finite abelian groups, {\em Amer. Math. Monthly} \textbf{114}(10):917--923, 2007.

\bibitem{Hor74a}
M.V.~Horo\v{s}evski\u{\i}, On automorphisms of finite groups, {\em Math. USSR Sb.} \textbf{22}(4):584--594, 1974.

\bibitem{Knu98a}
D.E.~Knuth, {\em The art of computer programming, Vol.~2, Seminumerical algorithms}, Addison-Wesley, Reading, MA, 3rd ed.~1998.

\bibitem{Lec94a}
P.~L'Ecuyer, Uniform random number generation, {\em Ann. Oper. Res.} \textbf{53}:77--120, 1994.

\bibitem{Lec12a}
P.~L'Ecuyer, Random number generation, in: {\em Handbook of computational statistics - concepts and methods. 1,2,} Springer (Springer Handbooks of Computational Statistics), Heidelberg, 2nd ed.~2012, pp.~35--71.

\bibitem{LN97a}
R.~Lidl and H.~Niederreiter, {\em Finite Fields}, Cambridge University Press (Encyclopedia of Mathematics and its Applications), Cambridge, 2nd ed.~1997.

\bibitem{LM72a}
H.~Liebeck and D.~MacHale, Groups with automorphisms inverting most elements, {\em Math. Z.} \textbf{124}:51--63, 1972.

\bibitem{Lie73a}
H.~Liebeck, Groups with an automorphism squaring many elements, {\em J. Austral. Math. Soc.} \textbf{16}:33--42, 1973.

\bibitem{LM73a}
H.~Liebeck and D.~MacHale, Groups of odd order with automorphisms inverting many elements, {\em J.~London Math.~Soc.~(2)} \textbf{6}:215--223, 1973.

\bibitem{Mac75a}
D.~MacHale, Groups with an automorphism cubing many elements, {\em J. Austral. Math. Soc.} \textbf{20}(2):253--256, 1975.

\bibitem{Mil29a}
G.A.~Miller, Groups which admit automorphisms in which exactly three-fourths of the operators correspond to their inverses, {\em Bull. Amer. Math. Soc.} \textbf{35}(4):559--565, 1929.

\bibitem{Mil29b}
G.A.~Miller, Possible $\alpha$-automorphisms of non-abelian groups, {\em Proc. Nat. Acad. Sci. U. S. A.} \textbf{15}(2):89--91, 1929.

\bibitem{Nie92a}
H.~Niederreiter, {\em Random number generation and quasi-Monte Carlo methods}, Society for Industrial and Applied Mathematics (CBMS-NSF Regional Conference Series in Applied Mathematics, 63), Philadelphia, PA, 1992.

\bibitem{Pot88a}
W.M.~Potter, Nonsolvable groups with an automorphism inverting many elements, {\em Arch. Math. (Basel)} \textbf{50}(4):292--299, 1988.

\bibitem{Rob96a}
D.J.S.~Robinson, {\em A Course in the Theory of Groups}, Springer (Graduate Texts in Mathematics, 80), New York, 2nd ed.~1996.

\bibitem{Zim90a}
J.~Zimmerman, Groups with automorphisms squaring most elements, {\em Arch. Math. (Basel)} \textbf{54}(3):241--246, 1990.

\end{thebibliography}
\end{document}